\documentclass[11pt]{article}
\usepackage[top=2.5cm,bottom=2.5cm,left=2.5cm,right=2.5cm]{geometry}
\usepackage{mathrsfs}
\usepackage{pstricks}
\usepackage{amsmath,amssymb,graphicx,amsthm,tabularx,array}
\usepackage{hyperref}
\usepackage{makecell}
\theoremstyle{plain}
\newtheorem{theorem}{Theorem}[section]

\newtheorem{coro}[theorem]{Corollary}
\newtheorem{lemma}[theorem]{Lemma}

\newtheorem{prop}[theorem]{Proposition}
\newtheorem{obsv}[theorem]{Observation}

\newtheorem{conj}[theorem]{Conjecture}

\theoremstyle{definition}
\newtheorem{example}[theorem]{Example}
\newtheorem{remark}[theorem]{Remark}

\newtheorem{other}{}

\newcommand{\one}{\hbox{1\hskip -4pt 1}}

\title{\bf On a relationship between the characteristic and matching polynomials of a uniform hypertree}

\author{
Honghai Li,%\footnote{lhh$@$mail.ustc.edu.cn},
\quad  Li Su%\footnote{suli$@$jxnu.edu.cn}
\\
\small   School of  Mathematics and Statistics, Jiangxi Normal University,  Nanchang\\
\small  Jiangxi 330022,  China. email: \tt{lhh$@$mail.ustc.edu.cn; suli$@$jxnu.edu.cn} \\
Shaun Fallat\footnote{Corresponding Author}
\\ %\quad Seyed Ahmad Mojallal\footnote{ahmad mojalal$@$yahoo.com}\\
\small Department of Mathematics and Statistics, University of Regina, Regina\\
\small Saskatchewan, S4S 0A2, Canada. email: \tt{shaun.fallat$@$uregina.ca}\\}
\date{}
\begin{document}
\maketitle
\begin{abstract}
A hypertree is a connected hypergraph without cycles. Further a hypertree is called an $r$-tree if, additionally, it is $r$-uniform. Note that 2-trees are just ordinary trees.  A classical result  states that for any 2-tree $T$ with characteristic polynomial $\phi_T(\lambda)$ and matching polynomial $\varphi_T(\lambda)$, then
$\phi_T(\lambda)=\varphi_T(\lambda).$
More generally, suppose $\mathcal{T}$ is an $r$-tree of size $m$ with $r\geq2$. In this paper, we extend the above classical relationship to $r$-trees and establish that
\[
\phi_{\mathcal{T}}(\lambda)=\prod_{H  \sqsubseteq \mathcal{T}}\varphi_{H}(\lambda)^{a_{H}},
\]
where the product is over all connected   subgraphs $H$ of $\mathcal{T}$, and the exponent $a_{H}$  of the factor $\varphi_{H}(\lambda)$ can be written as
\[
a_H=b^{m-e(H)-|\partial(H)|}c^{e(H)}(b-c)^{|\partial(H)|},
 \]
where $e(H)$ is the size of $H$, $\partial(H)$ is the boundary of  $H$, and $b=(r-1)^{r-1}, c=r^{r-2}$. In particular, for  $r=2$, the above correspondence reduces to the classical result for ordinary trees. In addition, we resolve a conjecture by Clark-Cooper [{\em Electron. J. Combin.}, 2018] and show that for any subgraph $H$ of an $r$-tree $\mathcal{T}$ with $r\geq3$,
$\varphi_H(\lambda)$ divides $\phi_{\mathcal{T}}(\lambda)$, and additionally
$\phi_H(\lambda)$ divides $\phi_{\mathcal{T}}(\lambda)$,
if either $r\geq 4$ or $H$ is connected when $r=3$.  Moreover, a counterexample is given for the case when $H$ is a disconnected subgraph of a 3-tree.

\vspace{3mm}


\vspace{2mm}


\end{abstract}

\section{Introduction}

Let $G$ be a simple, undirected  graph of order $n$. The characteristic polynomial, $\phi_{G}(\lambda)$, of   $G$   is the characteristic polynomial of its adjacency matrix $A(G)$, i.e. $\phi_{G}(\lambda)=\mathrm{Det}(\lambda I-A(G))$.  The matching polynomial, $\varphi_{G}(\lambda)$,  of   $G$   is defined to be
\[
\varphi_{G}(\lambda)=\sum\limits_{k\geq 0}(-1)^{k}m_{G}(k)\lambda^{n-2k},
\]
where $m_{G}(k)$ is the number of $k$-matchings in $G$. In Godsil~\cite{AlgComb}, these definitions suggest that the characteristic polynomial is an algebraic object and the matching polynomial a combinatorial one, however, these two polynomials are closely related and therefore have been studied together. A classical result concerning these polynomials states:
\begin{theorem}[\cite{AlgComb}]\label{forestcharmatchpoly}
If $G$ is a tree, then $\phi_{G}(\lambda)=\varphi_{G}(\lambda)$.
\end{theorem}

A hypergraph is a generalization of an undirected graph in which the edges are arbitrary subsets of the set of vertices. Tensors (also called hypermatrices) are natural generalizations of vectors and matrices.
 A vector can be represented
as an array indexed by single subscript, a matrix by two, whereas a higher order tensor
is indexed by a sequence of several subscripts.   The notion of
an eigenvalue of a tensor, as a natural generalization of an eigenvalue of a square
matrix, was proposed independently by Qi \cite{qi05} and Lim
\cite{Lim05}.
%Similarly as graphs and matrices in spectral graph theory, there is a natural one-to-one
%correspondence between hypergraphs and tensors.
Cooper and Dutle~\cite{CoopDut12} investigated the spectrum and characteristic polynomial of a uniform hypergraph based on the adjacency tensor. Further, Shao et al.~\cite{ShaoQiHu} provided some properties on the characteristic
polynomial of hypergraphs whose spectrum are $k$-symmetric.  Additionally, Clark and  Cooper~\cite{ClarkCoop21}  generalized the Harary-Sachs theorem to  uniform hypergraphs, and
Bao et al.~\cite{BaoFanWang} produced  a combinatorial method for computing characteristic polynomials of
starlike hypergraphs.  More recently, Chen and Bu~\cite{ChenBu}   gave a reduction formula for the characteristic polynomial of a uniform hypergraph with pendant edges and then derived an explicit expression for the characteristic polynomial and all distinct eigenvalues of uniform loose paths.
Finally, Zheng~\cite{complete3-graph} computed the characteristic polynomial of a complete 3-uniform hypergraph.

The aim of this paper is to generalize  Theorem~\ref{forestcharmatchpoly} to  uniform hypertrees. We begin with a careful review and survey of existing terms, notation, and necessary related results. In Section 3, we provide the framework for our analysis, and include a discussion on certain connections to digraphs, the process of {\em toppling}, and recurrent configurations. In Section 4, we state and establish the main purpose of this work by proving a relation between the matching polynomial and characteristic polynomial for $r$-trees. Finally, we close with some implications and applications of this work.

\section{Preliminaries}

In this section we provide all of the necessary background, including key terms and existing results that will be needed for the subsequent analysis.
\subsection{Resultants}

Given $n+1$ homogeneous polynomials $F_0, \ldots, F_n\in \mathcal{K}[x_0, \ldots, x_n]$ over a field $\mathcal{K}$, consider the system of equations
$$F_0(x_0, \ldots, x_n)=\cdots=F_n(x_0, \ldots, x_n)=0$$
and set $V(F_0, \ldots, F_n)=\{(a_0, \ldots, a_n)\in \mathcal{K}^{n+1}\,\,:\,\, F_i(a_0, \ldots, a_n)=0\,\,\mathrm{for\,\, all}\,\,0\leq i\leq n\}$, which is called the \textit{affine variety} (or simply a \textit{variety}, when there is no confusion with the projective variety) determined by $F_0, \ldots, F_n$.
Since each $F_i$ is homogeneous and of positive total degree, say $d_i$ and $F_i=\sum_{|\alpha|=d_i}c_{i, \alpha}x^{\alpha}$, the system of equations always has the (trivial) solution $x_0=\cdots=x_n=0$, and the surface defined by   $V(F_0, \ldots, F_n)$  has a conic shape so we employ the projective space $\mathbb{P}^n$ to account for the nontrivial solutions of this system of polynomials. Projectively, a system of $k$ polynomials in general position has a
$n-k$ dimensional space of solutions. If $k=n+1$, then   $n-k=-1$, which implies that for generic values of the coefficients, the system has no projective solutions. However, for certain values of the coefficients such a system can have more solutions than in general position. This is an important phenomenon called {\em degeneration}. For example, when the polynomials $F_i$   are all linear,  a degenerate system simply means that the determinant of the coefficient matrix vanishes. For each monomial $x^{\alpha}$ of degree $d_i$, we introduce a variable $u_{i, \alpha}$. Let $M$ be the total number of these variables, so that $\mathbb{C}^M$ is an affine space with coordinates $u_{i, \alpha}$ for all $0\leq i\leq n$ and $|\alpha|=d_i$. A point of $\mathbb{C}^M$ will be written  $(c_{i, \alpha})$. Then consider the ``universal" polynomials
 \begin{equation}\label{universalpoly}
\mathbf{F}_i=\sum_{|\alpha|=d_i}u_{i, \alpha}x^{\alpha},\qquad i=0, \ldots, n,
 \end{equation}
in the product $\mathbb{C}^M\times \mathbb{P}^n$. A point $(c_{i, \alpha}, a_0,\ldots, a_n)\in \mathbb{C}^M\times \mathbb{P}^n$ can be regarded as $n+1$ homogeneous polynomials and a point of $\mathbb{P}^n$. The ``universal" polynomials $\mathbf{F}_i$ are actually polynomials on $\mathbb{C}^M\times \mathbb{P}^n$, which produce the subset $W=\mathbf{V}(\mathbf{F}_0, \ldots, \mathbf{F}_n)$. Consider the natural projection map $\pi: \mathbb{C}^M\times \mathbb{P}^n \rightarrow \mathbb{C}^M$ defined by $\pi(c_{i, \alpha}, a_0,\ldots, a_n)=(c_{i, \alpha})$, and under this projection, the variety $W\subset \mathbb{C}^M\times \mathbb{P}^n$ maps to
\begin{eqnarray*}
\pi(W) &=& \{(c_{i, \alpha})\in \mathbb{C}^M: ~\exists (a_0,\ldots, a_n)\in \mathbb{P}^n \,\, \mathrm{such\, that}\,\, (c_{i, \alpha}, a_0,\ldots, a_n)\in W\},
\end{eqnarray*}
which coincides with the collections of equations $F_0=\cdots=F_n=0 $ with degrees $d_0, \ldots , d_n$ that has a nontrivial solution.
By the Projective Extension Theorem (see, for example, Theorem 6 in \S5 of Chapter 8 of \cite{CoxLittleShea-ideals}), $\pi(W)$ is a variety in $\mathbb{C}^M$ and thus the existence of a nontrivial solution of $F_0=\cdots=F_n=0$ is determined by polynomial conditions on the coefficients of $F_0, \ldots, F_n$. Furthermore, it was known that the variety $\pi(W)$ is irreducible of dimension $M-1$ and hence it must be defined by exactly one irreducible equation as stated concretely in the following well-known result.

 \begin{prop}[\cite{CoxLittleShea}]
If we fix positive degrees $d_0,\ldots, d_n$, then there is a
unique polynomial  $\mathrm{Res}\in \mathbb{Z}[u_{i, \alpha}]$ which satisfies the following properties:
 \begin{enumerate}
   \item If $F_0, \ldots, F_n\in \mathbb{C}[x_0, \ldots, x_n]$ are homogeneous of degrees $d_0,\ldots, d_n$,
then the equations $F_0=\cdots=F_n=0$ have a nontrivial solution over $\mathbb{C}$ if and only if
$\mathrm{Res}(F_0, \ldots, F_n) = 0$;
   \item  $\mathrm{Res}(x_0^{d_0}, \ldots, x_n^{d_n}) = 1$;
   \item $\mathrm{Res}$ is irreducible, even when regarded as a polynomial in $\mathbb{C}[u_{i, \alpha}]$.
 \end{enumerate}
\end{prop}
The unique irreducible polynomial $\mathrm{Res}\in \mathbb{Z}[u_{i, \alpha}]$   is called the \textit{resultant} of degrees $(d_0,\ldots, d_n)$, and
 $\mathrm{Res}(F_0, \ldots, F_n)$ is understood as the evaluation of $\mathrm{Res}$ at the point $\{u_{i, \alpha}=c_{i, \alpha}\}$ with $c_{i, \alpha}$ given by  $F_i$.

We require a well-known method to construct a formula for the resultant, whose essence is to multiply each equation by appropriate monomials until we arrive at a determinant of an associated square matrix. Suppose we have $F_0, \ldots, F_n\in \mathbb{C}[x_0, \ldots, x_n]$  with total degrees $d_0,\ldots, d_n$. Then set
\[
d=\sum_{i=0}^n(d_i-1)+1.
\]

Now let $S^{(n, d)}$ denote  the set consisting of all monomials $x^{\alpha}=x_0^{a_0}x_1^{a_1}\cdots x_n^{a_n}$ of total degree $d$ and define the $n+1$ disjoint subsets  $S_0, S_1, \ldots, S_n$ of these monomials by
 \begin{equation}\label{Si}
 S_i=\{x^{\alpha}: |\alpha|=d,\,\, x_0^{d_0},\ldots, x_{i-1}^{d_{i-1}}\, \mathrm{does\,\, not \,\, divide}\, x^{\alpha}, \,\, \textrm{but}\,\, x_i^{d_i}\,\,\mathrm{does}\}.
 \end{equation}
 Note that $S_0, S_1, \ldots, S_n$ constitutes a partition of the set $S^{(n, d)}$. Consider the  system of equations
\begin{eqnarray}\label{systm_equatns}
  x^{\alpha}/x_0^{d_0}\cdot F_0 &=& 0\,\, \textrm{for} \,\,\textrm{all} \,\, x^{\alpha}\in S_0\nonumber \\
    & \vdots &   \\
x^{\alpha}/x_n^{d_n}\cdot F_n &=& 0\,\, \textrm{for} \,\,\textrm{all} \,\, x^{\alpha}\in S_n. \nonumber
\end{eqnarray}

Note that the polynomials $x^{\alpha}/x_i^{d_i}\cdot F_i$ ($x^{\alpha}\in S_i$) are homogeneous of same  degree $d$, and then each polynomial on the left side of \eqref{systm_equatns} can be written as a linear combination of monomials of total degree $d$. Suppose that there are $N$ such monomials with  $N=\binom{d+n}{n}$. Thus, regarding the monomials of total degree $d$ as unknowns, we arrive at a system of $N$ linear equations in $N$ unknowns. The determinant of the coefficient matrix of the $N\times N$ system of equations given by \eqref{systm_equatns} is denoted by $D_n$. Then $D_n$ is a polynomial in the coefficients of the $F_i$, and is homogeneous in the coefficients of  $F_i$ of degree equal to the cardinality of $S_i$, for any fixed $i$ between 0 and $n$. Observe that $D_n$ vanishes whenever $F_0= \cdots= F_n=0$ has a nontrivial solution, which implies that $D_n$ vanishes on $\mathbf{V}(\textrm{Res})$  and thus
$$D_n=\textrm{Res}\cdot\, (\textrm{extraneous}\,\, \textrm{factor}),$$
where  the extraneous factor is an integer polynomial in the coefficents of $\bar{F}_0, \ldots, \bar{F}_{n-1}$, where $\bar{F}_i=F_i(x_0,\ldots, x_{n-1},0)$. For this process, if we fix $i$ between 0 and $n-1$ and order the variables so that $x_i$ comes last, then we may produce slightly different sets $S_0, \ldots, S_n$ and a slightly different system of equations \eqref{systm_equatns}. We will let $D_i$ denote the  determinant of this system of equations. (Note that there are many different orderings of the variables for which $x_i$ is last.  We pick one when computing $D_i$.) Note that $D_i$ is  homogeneous in the coefficients of each $F_j$ and particularly, is homogeneous of degree $d_0\cdots d_{i-1}d_{i+1}\cdots d_n$ in the coefficients of $F_i$.

 \begin{prop}[\cite{CoxLittleShea}]\label{prop_GCD}
When    $\mathbf{F}_0, \ldots, \mathbf{F}_n$ are universal polynomials as in \eqref{universalpoly}, the resultant is the greatest common divisor of the   polynomials   $D_0, \ldots, D_n$ in the ring $\mathbb{Z}[u_{i, \alpha}]$, i.e.,
\[
\mathrm{Res}=\pm \mathrm{GCD}(D_0, \ldots, D_n).
\]
\end{prop}

%Let $d_0,\ldots, d_n$ and $d$ be as usual.
A monomial $x^{\alpha}$ of total degree $d$ is \textit{reduced}  if $x_i^{d_i}$ divides $x^{\alpha}$ for exactly one $i$. Let $D_n'$ be the determinant of the submatrix of the coefficient matrix of \eqref{systm_equatns} obtained by deleting all rows and columns corresponding to   reduced monomials.

\begin{prop}[\cite{CoxLittleShea}]\label{prop_ratiodet}
When    $\mathbf{F}_0, \ldots, \mathbf{F}_n$ are universal polynomials, the resultant is given by
\[
\mathrm{Res}=\pm \frac{D_n}{D_n'}.
\]
 Further, if $\mathcal{K}$ is any field and  $F_0, \ldots, F_n\in \mathcal{K}[x_0, \ldots, x_n]$, then the above formula for $\mathrm{ Res}$ holds whenever
 $D_n'\neq0$.
\end{prop}

\subsection{ Hypergraphs and tensors}

A \textit{hypergraph}  $\mathcal{H}$ is a pair $(V,E)$, where $V=V(\mathcal{H})$ is a finite set and $E\subseteq \mathcal{P}(V)$ with $\mathcal{P}(V)$ being the power set of $V$.
%The elements of $V$  are referred to as \emph{vertices} and the elements of $E$ are called \emph{edges}. An \textit{element} of $\mathcal{H}$ is a vertex or an edge of it, that is, any element of $V\cup E$.
 When every edge $e\in E(\mathcal{H})$ contains precisely $r$ vertices, $\mathcal{H}$ is said to be \emph{$r$-uniform}, or simply an $r$-graph.   For a vertex $v\in V$, let   $E_v(\mathcal{H})$ (or simply $E_v$) denote the set of edges containing $v$. The cardinality $|E_v|$ is the \emph{degree} of $v$, denoted by $\deg(v)$.
  If any two edges in $\mathcal{H}$ share at most one vertex, then $\mathcal{H}$ is said to be a \emph{linear hypergraph}. For any subset $E'\subseteq E$, the hypergraph with edge set $E'$ and vertex set $U=\{v: \exists \, e\in E' \,\,\mathrm{such\,\, that}\,\, v\in e\}$ is called a \textit{subgraph} of $\mathcal{H}$ (induced by $E'$),  and denoted by $\mathcal{H}[E']=(U, E')$.
   For any subset $V'\subseteq V$, the hypergraph   with vertex set $V'$ and edge set $F=\{e: e\in E  \,\,\mathrm{and}\,\, e\subseteq V'\}$ is called a \textit{subgraph} of $\mathcal{H}$ (induced by $V'$), and denoted by $\mathcal{H}[V']=(V', F)$.  We may use $\mathcal{H}-V'$  and $\mathcal{H}-E'$ to denote the subgraphs of $\mathcal{H}$ induced by the vertex set $V\setminus V'$ and by the edge set $E\setminus E'$, respectively. When  $V'=\{v\}$ or $E'=\{e\}$, $\mathcal{H}-V'$  and $\mathcal{H}-E'$ are simply written as $\mathcal{H}-v$  and $\mathcal{H}-e$, respectively.
   Henceforth,   we keep  the term ``subgraph"   instead of   ``subhypergraph" for ease of reading. This convention will also apply to other basic terminology involving hypergraphs, etc.
We write $\mathcal{H}' \sqsubseteq\mathcal{H}$ to mean $\mathcal{H}'$ is a connected subgraph of $\mathcal{H}$.
   For a given subgraph $\mathcal{H}'$ of $\mathcal{H}$, a   \textit{boundary edge} of $\mathcal{H}'$ is an edge of $\mathcal{H}$ which contains vertices both from $V(\mathcal{H}')$ and from $V(\mathcal{H}) \setminus V(\mathcal{H}')$, and the \textit{boundary} of $\mathcal{H}'$, denoted $\partial(\mathcal{H}')$, is defined to be  the set of boundary edges of $\mathcal{H}'$.  For more see \cite{Berge-hypergraph, Bretto, zykov}.

 In  a hypergraph $\mathcal{H}$,  two vertices $u$ and $v$ are {\em adjacent} if there is an edge $e$ of $\mathcal{H}$ such
that $\{u,v\}\subseteq e$, and in this case  $e$ is   {\em incident} to  $v$ (and $u$), and $(v, e)$ is called an \textit{incidence} of $\mathcal{H}$ (determined by $e$).
A {\em walk}  in $\mathcal{H}$ is a finite sequence composed
of elements (i.e. vertices or edges) such that vertices and edges alternate and any
two neighbouring elements are incident. The beginning and end of a walk
can be, independently of each other, either a vertex or an edge.
The walk is {\em  closed} if its beginning and end coincide.
A \textit{path} $P$ (of length $s$) in $\mathcal{H}$ from a vertex $x$ to a vertex $y$ is a walk $x e_1v_1e_2\cdots v_{s-1} e_{s}y$ in which  $x, v_1,v_2, \ldots v_{s-1}$ are distinct vertices and $e_1, e_2, \ldots,  e_{s}$ are distinct edges of $\mathcal{H}$. If this path $P$ is closed and $s>1$, then it is called a \textit{cycle of length} $s$. A path  in $\mathcal{H}$ from an element   to another  is defined similarly.
A
hypergraph $\mathcal{H}$ is called {\em connected} if for any vertices $u$, $v$, there is a path connecting
$u$ and $v$, with the convention that a hypergraph  consisting of a single vertex is connected. The \textit{distance} $\mathrm{dist}(x, y)$ between two vertices $x$ and $y$ is the minimum length of a path between $x$ and $y$.
A   \textit{hypertree} is a connected hypergraph containing no cycles, and is called an $r$-tree, if, in addition, it is $r$-uniform. Note that a hypertree must be a linear hypergraph.

A basic result on the order and size of a hypertree needed later is presented as follows.

\begin{lemma}[\cite{Berge-hypergraph}]\label{hypertree-order-size}
Suppose $\mathcal{H}$ is an $r$-tree, where $n$  and $m$  are the number of vertices and edges, respectively. Then
\[
m=\frac{n-1}{r-1}.
\]
 \end{lemma}

For positive integers $r$ and $n$, a
{\em tensor} $\mathcal{A}=(a_{i_1i_2\cdots i_r})$ of order $r$ and dimension $n$
refers to a multidimensional array   with entries
$a_{i_1i_2\cdots i_r}$ such that $a_{i_1i_2\cdots i_r}\in\mathbb{C}$ for
all $i_1$, $i_2$, $\ldots$, $i_r\in[n]$, where $[n]=\{1,2,\ldots,n\}$.

Let $\mathcal{A}$ be an order $r$ and dimension $n$ tensor. If there exists  $\lambda\in\mathbb{C}$
and a nonzero vector $x\in\mathbb{C}^{n}$ such that
\begin{equation*}
\mathcal{A}x=\lambda x^{[r-1]},
\end{equation*}
where $\mathcal{A}x$ is an $n$-dimensional vector with $ \sum_{i_2,\ldots,i_r=1}^na_{ii_2\cdots i_r}x_{i_2}\cdots x_{i_r}$ as its $i$-th entry,
and  $x^{[r-1]}$ is a vector with $i$-th entry $x^{r-1}_i$, then $\lambda$ is called an {\em eigenvalue} of $\mathcal{A}$, $x$ is called
an {\em eigenvector} of $\mathcal{A}$ corresponding to the eigenvalue $\lambda$ (see \cite{qi05, QiLuo-2017}). Let $\mathcal{I}$ denote the identity tensor with the same order and dimension as $\mathcal{A}$, i.e. $\mathcal{I}_{i_1i_2\cdots i_r}=1$ if $i_1=i_2=\cdots =i_r$ and is 0 otherwise. Then $(\lambda \mathcal{I}-\mathcal{A})x=0$ is a system of equations as follows
\begin{equation}\label{systemequ}
F_i:=\lambda x_i^{r-1}- \sum_{i_2,\ldots,i_r=1}^na_{ii_2\cdots i_r}x_{i_2}\cdots x_{i_r}=0,\,\, \mathrm{for}\,\, i=1,\ldots, n.
\end{equation}
 The resultant $\mathrm{Res}(F_1,\ldots, F_n)$ is called the  \textit{characteristic polynomial of tensor} $\mathcal{A}$ and is denoted by $\phi_{\mathcal{A}}(\lambda)$. A fundamental result by Qi~\cite{qi05} states that a scalar $\lambda$ is an eigenvalue of  $\mathcal{A}$ if and only if it is a root of the characteristic polynomial of  $\mathcal{A}$.

Let $\mathcal{H}=(V, E)$ be an $r$-graph wtih $V=[n]$. The adjacency
tensor of $\mathcal{H}$ is defined as the order $r$ and dimension $n$ tensor
$\mathcal{A}(\mathcal{H})=(a_{i_1i_2\cdots i_r})$, whose $(i_1i_2\cdots i_r)$-entry is
\[
a_{i_1i_2\cdots i_r}=\begin{cases}
\frac{1}{(r-1)!}, & \text{if}~\{i_1,i_2,\ldots,i_r\}\in E,\\
0, & \text{otherwise}.
\end{cases}
\]
 The eigenvalues and  characteristic polynomial of $\mathcal{A}(\mathcal{H})$ are called the eigenvalues and  characteristic polynomial of  $\mathcal{H}$, respectively.

We include some notation for convenience. For any function (or vector) $x=(x_i)_{i=1}^n$ on $V$ and $e=(i_1,\ldots, i_m)\in V^m$ ($m$ is a positive integer), $\alpha=(a_1, \ldots, a_{n})\in \mathbb{N}^{n}$,  we denote $x_e=x_{i_1}x_{i_2}\cdots x_{i_m}$, $x|_e=(x_{i_1},x_{i_2},\cdots, x_{i_m})$ and  $x^{\alpha}=x_1^{a_1}x_2^{a_2}\cdots x_n^{a_n}$. Then $(\lambda, x)$ is an eigenpair of $\mathcal{H}$ if and only it satisfies

\begin{equation*}
\lambda x_i^{r-1}=\sum_{i_2,\ldots,i_r\in V}a_{ii_2\cdots i_r}x_{i_2}\cdots x_{i_n},\,\, \mathrm{for}\,\, i=1,\ldots, n,
\end{equation*}
or equivalently,
\begin{equation*}
\lambda x_i^{r-1}=\sum_{e\in E_{i}}x_{e\setminus\{i\}},\,\, \mathrm{for}\,\, i=1,\ldots, n.
\end{equation*}

    \textit{A matching} of a hypergraph $\mathcal{H}$ is a set of mutually
disjoint (or independent) edges in $E$. A \textit{$k$-matching} is  a matching
consisting of $k$ edges. We denote by $m_{\mathcal{H}}(k)$ the number of
$k$-matchings in $\mathcal{H}$.  The  \textit{matching number} $\nu(\mathcal{H})$ of $\mathcal{H}$ is the maximum cardinality of a matching.

Su et al.~\cite{SuKLS} defined the \textit{matching polynomial} for a general $r$-uniform hypergraph $\mathcal{H}$ of order $n$  as
\begin{equation}\label{e-matchingpoly}
\varphi_{\mathcal{H}}(x)=\sum\limits_{k\geq0}(-1)^{k}m_{\mathcal{H}}(k)x^{n-kr},
\end{equation}
and introduced this polynomial to investigate a perturbation on the spectral radius of hypertrees. Some fundamental properties of this polynomial were obtained  in \cite{SuKLS} and the next result, which is applied later is listed as follows.

 \begin{theorem}[\cite{SuKLS}]\label{thm_matchingpoly}
The matching polynomial for the disjoint union of the hypergraphs  $\mathcal{H}_1, \mathcal{H}_2, \ldots, \mathcal{H}_s$ is equal to the product $\varphi_{\mathcal{H}_1}(x)\varphi_{\mathcal{H}_2}(x) \cdots \varphi_{\mathcal{H}_s}(x)$.  \end{theorem}

When $r=2$, $\varphi_{\mathcal{H}}(x)$ is the matching polynomial of a graph. Thus this definition  may be viewed as a natural generalization of a matching polynomial to hypergraphs.

\subsection{A chip-firing game on graphs}

Let $G=(V, E)$ be a connected undirected simple graph. A \textit{configuration} on $G$ is any nonnegative integer-valued  function $s:V\rightarrow  \mathbb{N}$, recording the number of chips located at each vertex of $G$.
  A   vertex $v$ is said to be {\em  stable} for $s$ if  $s(v)<\deg(v)$ and \textit{unstable} otherwise, and  a configuration $s$ is   {\em stable} if every vertex of $G$ is stable for $s$.
The \textit{chip-firing process}  on $G$ begins with an initial configuration. At each step,  an unstable vertex $v$ is selected to {\em fire}, that is, a certain number of chips are selected and move from $v$ to its adjacent vertices,  one chip
distributed along each edge incident to $v$. If, at any stage, a stable configuration is reached, the process stops.  In the chip-firing process, there may be multiple vertices ready to fire, however the order in which vertices are fired is not important, more precisely, the final configuration reached  is invariant under the order of firings, which is a well-known confluence property of such chip-firing processes.

\begin{theorem}[\cite{BjoLovaShor}]\label{confluencenoroot}
Given a connected graph and an initial configuration $s$, either every
chip-firing process starting from $s$ can be continued infinitely, or every
chip-firing process from $s$ terminates after the
same number of moves with the same final configuration.
\end{theorem}

%\begin{theorem}[\cite{BjoLovaShor}]
 % Let $G$ be a graph with $n$ vertices and $m$ edges,  and  the sum of chips on all vertices is  $N$. We have
 % \begin{enumerate}
 %   \item if $N>2m-n$, then the game  is infinite.
 %   \item if $m \leq N\leq  2m-n$, then there exists an initial configuration guaranteeing finite
%termination and also one guaranteeing infinite game.
   % \item if $N<m$, then the game is finite.
 % \end{enumerate}
%\end{theorem}

Now a natural modification of this   process is considered by introducing a root vertex.
Let $G=(V, E)$ be a connected graph  and let $q\in V$ be a distinguished vertex called the \textit{root}.
A \textit{configuration} on $G$ with root $q$ is any nonnegative  integer-valued  function on $V\setminus\{q\}$. All terminologies apply to $G$ with a  root, such as stable vertex and stable configuration.  In this case the chip-firing process on $G$ with a root $q$ starts with an initial  configuration $s$ and at each step a non-root vertex that is unstable is selected and fires. If, at any step, a stable configuration is reached, then (and only then) the root vertex is fired.
A sequence of firing  is \textit{legal} if and only if each occurrence of a non-root vertex $v$ follows a configuration $s$ with  $s(v)\geq \deg(v)$ and each occurrence of the root follows a stable configuration.
A configuration  $s$   on $G$ is called \textit{recurrent}  if there is a legal sequence of steps for $s$ which reaches the same  configuration, and is   \textit{critical} if it  is both stable and recurrent.

 \begin{theorem}[\cite{Biggs-criticalgroup, sandpilegroupdualgraph, chip-firing}]\label{confluenceroot}
Let $G$ be a connected graph with a root.  Then any configuration on $G$ will reach a unique critical configuration after a legal chip-firing process.
\end{theorem}

The set of critical configurations on $G$  with a root  can be given the structure of an abelian group, of which the order   is equal to the number of spanning trees of $G$.

 A sequence $(a_1, a_2, \ldots a_n)$ is  called a {\em parking function of length $n$} if
 there is a permutation $\sigma$ on $[n]$ such that $1\leq a_{\sigma(i)} \leq i$. Let $\mathcal{P}_n$ be the set of all parking functions of length $n$. As usual, the complete 2-graph on $n$ vertices is denoted $K_n$.

 \begin{lemma}[\cite{sandpilegroupdualgraph}]\label{prop_parkingfunc}
 A configuration $(a_1, \ldots, a_n)$ of the complete graph $K_{n+1}$ on vertex set $[n]_0$ (with 0 as the root) is a critical configuration if and only if $(n-a_1, n-a_2, \ldots n-a_n)$ is a parking function of length $n$.
\end{lemma}

Using Cayley's Formula for the number of spanning trees in a complete graph, together with the above lemma, it follows that $|\mathcal{P}_n|=(n+1)^{n-1}$.

\section{A reduction to digraphs}

A \textit{directed graph} (or just \textit{digraph})  $D$ consists of a non-empty finite set
$V (D)$ of elements called \textit{vertices} and a finite set $E(D)$ of ordered pairs of
distinct vertices called \textit{arcs}. Note that   parallel (also called multiple) arcs   or loops are not permitted  by definition.
An arc $(u, v)$ in $D$ is usually denoted  by $u\rightarrow v$.
 In $D$ a vertex $y$ is \textit{reachable} from a vertex $x$ if there is a directed path in $D$ from $x$ to $y$, and $y$ is \textit{strongly connected} to $x$ if either is reachable from the other.
 The distance from $x$ to $y$ in  $D$,  denoted $\mathrm{dist}(x, y)$, is the minimum length of a directed path from $x$ to $y$,
if $y$ is reachable from $x$;  otherwise $\mathrm{dist}(x, y)=\infty$.
A digraph is \textit{strongly connected} (or, just, \textit{strong})  if every vertex is reachable from every other vertex in $D$.
 A \textit{strong component} of a digraph $D$ is a maximal induced subdigraph
of $D$ which is strong. For a digraph  $D$ with vertices  $v_1,\ldots, v_n$, the \textit{adjacency matrix} $A(D) =(a_{ij})$
of   $D$  is an $n\times n$ matrix such that $a_{ij}=1$  if $v_i\rightarrow v_j$ and $a_{ij}=0$
otherwise.  Next we shall be concerned with the characteristic polynomial of $D$, whose coefficients can be computed via the collection of all cycles in $D$ as follows.

\begin{prop}[\cite{theorygraphspectra}]\label{prop_charpoly_digraph}
Let
\[\phi_D(\lambda)=\mathrm{Det}(\lambda I-A)=\lambda^n+a_1\lambda^{n-1}+\cdots+a_n\]
 be the characteristic polynomial of an arbitrary digraph $D$ with $A$ as its adjacency matrix. Then
\[
a_i= \sum_{L\in \mathfrak{L}_i}(-1)^{p(L)} \qquad (i=1,\ldots ,n),
\]
where $\mathfrak{L}_i$ is the set of all linear directed subgraphs $L$ of $D$  with exactly $i$ vertices; $p(L)$ denotes the number of components of $L$ (i.e., the number of cycles to which $L$ is composed.)
\end{prop}

Note that the characteristic polynomial of   digraph $D$ is the product of the characteristic polynomials corresponding to the principal submatrices indexed by the collection of all strong  components of $D$.

\subsection{A toppled digraph}

Let  $\mathcal{H}$ be an $r$-graph on a vertex set $V=\{v_0, v_1, \ldots, v_n\}$.
 Given any ordering $ \succ$ on $V$, and any function $x=(x_v)$  on $V$, there is a natural ordering on the variables that is consistent with the ordering $\succ$ as follows
 \[
 x_u>x_v\quad \mathrm{whenever}\quad u\succ v.
 \]
%Then the orderings on the vertices and variables are consistent and will not be specified on whichever henceforth.

Choose an ordering on the variables, say $x_0>x_1>\cdots> x_n$.  The characteristic polynomial $\phi_{\mathcal{H}}(\lambda)=\mathrm{Res}(F_0, F_1, \ldots, F_n)$, where $F_i:=\lambda x_{v_i}^{r-1}-\sum_{e\in E_{v_i}}x_{e\setminus\{v_i\}}$.
As in \eqref{systm_equatns}, regarding the monomials of total degree $d$ (note that $d=(r-2))(n+1)+1$ now) as unknowns, we have a corresponding system of $N$ linear equations in $N$ unknowns.
Recall the determinant of the coefficient matrix associated with this linear system  is $D_n$. Similarly the determinant $D_i$ can be obtained when $x_i$ ordered as the last variable. For our purpose we are only concerned with $D_n$. A key point needed for this reduction to a digraph is the following.

\begin{obsv}
The determinant $D_n$ can be viewed as the characteristic polynomial of a directed graph $\mathfrak{D}_n$ which has $S^{(n, d)}$, the set of all monomials of total degree $d$, as its vertex set, and
\begin{itemize}
  \item  for any $x^{\alpha}, x^{\beta} \in S^{(n, d)}$ and  assume that $x^{\alpha} \in S_k$, $x^{\alpha}$ is connected to $x^{\beta}$ by an arc if and only if $x^{\beta}/x^{\alpha}=x_e/x_{v_k}^{r}$ for some incidence $(v_k, e)$ in  $\mathcal{H}$.
\end{itemize}
\end{obsv}
Note that a vertex $x^{\alpha}$ in  $\mathfrak{D}_n$ which belongs to $S_k$  has outdegree   equal to  the degree of vertex $v_k$ in $\mathcal{H}$.

Each monomial $x^{\alpha}=x_0^{a_0}x_1^{a_1}\cdots x_n^{a_n}$ of total degree $d$ can be viewed as an assignment of chips on vertices of $\mathcal{H}$, i.e., assigning a total sum of $d$ chips to $\mathcal{H}$ and exactly $a_i$ chips is assigned to each vertex $v_i$. If  $x^{\alpha} \in S_k$ and $x^{\beta}/x^{\alpha}=x^e/x_{v_k}^{r}$ for some incidence $(v_k, e)$ in  $\mathcal{H}$, then  we  say that $x^{\beta}$ is obtained from $x^{\alpha}$ by toppling   the incidence $(v_k,e)$. By toppling  an incidence $(v, e)$ we mean that chips from $v$ are redistributed to the adjacent vertices as follows:  $r-1$ chips are taken from vertex $v$ and one chip is added to each of $r-1$ neighbors of $v$ in $e$, the assignments of all remaining vertices are unchanged.  For convenience,   we also say that vertex {\em $v$ is toppled inside edge $e$}. We note that when vertex $v$ is toppled, any incident edge can be chosen to topple and thus there are $\deg(v)$ such choices.
 A vertex $v$ can only be toppled (inside some edge $e$ incident to $v$) in $r$-graph $\mathcal{H}$ according to the toppling rule as follows:
\begin{itemize}
    \item vertex  $v$ has  at least $r-1$ chips   but any vertex (according to the given ordering)  larger than $v$ does not.
  \end{itemize}

A vertex with at least $r-1$ chips is said to be an {\em unstable} vertex of $\mathcal{H}$, and by the toppling rule,  when there is more than one unstable vertex, only the largest one of them  should be toppled. For a given configuration $\alpha$ on an $r$-graph $\mathcal{H}$ and for an edge $e$ in $\mathcal{H}$, we say the incidences determined by $e$ can be toppled  {\em contiguously} starting from $\alpha$ if with respect to
$\alpha$, some vertex in  $e$ is unstable and any unstable vertex not in $e$ is less than all vertices in $e$.   Consequently we have the following result, whose proof is simple and is omitted.

\begin{lemma}\label{lem_unique-topple-order}
If  the incidences determined by an edge $e$ occur contiguously, then  the ordering of toppling such incidences determined by $e$ is unique.
\end{lemma}

In order to study toppling on a hypergraph, we consider the chip-firing process on a labelled graph (without a  root),  on which
 an ordering of the vertices is given. Then for the chip-firing process on labelled graph, when there are multiple vertices ready to fire, only the largest labelled vertex can be fired. Consequently the firing order is unique.

For an $r$-graph  $\mathcal{H}$ with an initial configuration $\alpha$,  if   the incidences determined by an edge $e=\{i_1, i_2, \ldots, i_r\}$ can be toppled contiguously from $\alpha$, then the topplings on $\mathcal{H}$ in $e$ reduce to a chip-firing process on the labelled complete 2-graph $K_{r}$ with vertex set $\{i_1, i_2, \ldots, i_r\}$. Assume that $i_r\succ i_{r-1}\succ\cdots\succ i_1$ in  $\mathcal{H}$, and that this ordering is adopted on the set $V(K_{r})$.
Consider chip-firing on $K_{r}$ with an initial configuration such that each vertex $i_j$ has $\alpha(i_j)$ chips for $j=1, \ldots, r$.
The vertex $i_1$ is distinguished since on the one hand,  it can be fired only when $i_2, \ldots, i_r$ are all stable so it is natural that $i_1$ serves as a root; on the other hand,  it behaves as a non-root vertex since it can be fired only when it has at least $r-1$ chips.
Thus, in a sense  the chip-firing process on labelled graph evolves in a manner just like both types of chip-firing processes on the unlabelled graph, namely, one with a root and one without a root. Furthermore, if the chip-firing process on  $K_{r}$ does not terminate, it must reach a unique critical configuration as implied by the
confluence property of chip-firing (see Theorems~\ref{confluencenoroot} and \ref{confluenceroot}).

%Because the vertex $i_r$ can be fired only when $i_2, \ldots, i_r$ are all stable, it seems valid to serve as a root; and because $i_r$ can be fired only when it has at least $r-1$ chips, it behaves   like a non-root vertex meanwhile. Thus the chip-firing process on $K_{r}$ in this case has  the flavours of   chip-firings on  $K_{r}$ when having no root and when  having a root.
%When a chip-firing on  $K_{r}$ never terminates, it will reach a unique critical configuration on $K_{r}$ by the confluence property of chip-firing (see Theorems~\ref{confluencenoroot}and \ref{confluenceroot}).

Although digraphs $\mathfrak{D}_0, \mathfrak{D}_1, \ldots, \mathfrak{D}_n$ all have the same vertex set $S^{(n, d)}$, they have different arc sets.
Each $\mathfrak{D}_i$ is called a \textit{toppled digraph} of $\mathcal{H}$ (rooted at $v_i$, for specifically). For example, if $\mathcal{H}$ is  a 3-graph of order 5, then $d=6$ and a monomial $x_0^2x_3^4$ belongs to $S_3$  in $\mathfrak{D}_0$ while it belongs to $S_0$  in $\mathfrak{D}_3$.

\subsection{Recurrent configurations}

There is an obvious one-to-one correspondence between monomials $x^{\alpha}$  and vectors $\alpha$ from which most of the terminology for one also applies to the other. For an $r$-graph $\mathcal{H}$ of order $n+1$, the vector $\alpha=(a_0, a_1, \ldots,a_n)$ corresponding to a monomial $x^{\alpha}=x_0^{a_0}x_1^{a_1}\cdots x_n^{a_n}$
 of total degree $d$  is called a {\em configuration} on  $\mathcal{H}$, which can also be viewed as a function defined on the vertices of  $\mathcal{H}$ such that $a_i$ is the value of the function at vertex $v_i$, i.e. $\alpha(v_i)=a_i$. Suppose that $\mathcal{Z}$ is a non-empty finite sequence of (not necessarily distinct) incidences of  $\mathcal{H}$, such that starting from configuration $\alpha$, the incidences can be toppled in the order of $\mathcal{Z}$ (called a \textit{toppling sequence}). If $(v,e)$ occurs $z(v,e)$ times, we shall refer to the vector $z$,   indexed by all incidences, as the representative vector for $\mathcal{Z}$. The configuration $ \alpha'$ obtained from a sequence of topplings is given by
\[
\alpha'(v)=\alpha(v)-(r-1)\sum_{e\in E_v}z(v,e)+\sum_{f\in E_v, v\neq w\in f}z(w, f).
\]

The relationship between $ \alpha$ and $ \alpha'$ can be written more conveniently based on the {\em generalized incidence matrix} $B$, whose rows are indexed by vertices and columns are indexed by incidences, as follows:
\[
B_{v, (u,e)}=\left\{
                 \begin{array}{ll}
                   r-1, & \hbox{$v=u\in e$;} \\
                   -1, & \hbox{$v, u\in e,v\neq u$;} \\
                   0, & \hbox{otherwise.}
                 \end{array}
               \right.
\]

In terms of $B$ the relationship between $ \alpha$ and $ \alpha'$ is given by
\[
\alpha'=\alpha-Bz.
\]

Note that all row and column sums of $B$ are equal to zero.

\begin{lemma}\label{lem_rank}
For any connected hypergraph of order $n+1$ with associated generalized incidence matrix $B$, the rank of $B$  is  equal to $n$.
\end{lemma}

\begin{proof}
Permute the columns of $B$ and partition the columns of $B$ as    $B=(B_{e_1}|B_{e_2}|\cdots |B_{e_m})$, where the $B_{e_i}$ denote the whole-rows submatrix of $B$ with columns indexed by all incidences determined by $e_i$. For any $(n+1)$-dimensional vector $z$, if $z^TB=0$ then $z^TB_{e_i}=0$  for all $i=1,\ldots, m$.  We may assume $B_{e_1}=(rI-J|O)^T$ because we can permute the rows of $B$ appropriately as needed, where $I$ and $J$ are the identity matrix and all-ones matrix of  order $r$, respectively.  Thus $(z_1,\ldots, z_r)^T(rI-J)=0$ and this implies that $z_1=\cdots=z_r$. Since the hypergraph $\mathcal{H}$ is connected, there exists some edge, say $e_2$ (=$\{i_1, i_2, \ldots, i_r\}$), such that $e_1\cap e_2\neq\emptyset$. Similarly it can be shown that $z_{i_1}=\cdots=z_{i_r}$. Further from $e_1\cap e_2\neq\emptyset$ it follows that   $z_1=\cdots=z_r=z_{i_1}=\cdots=z_{i_r}$ and so continuing in this way we know that  entries of $z$ are all the same.
\end{proof}

For an edge $e$ in $\mathcal{H}$, use $j_e$ to denote  the characteristic vector of the subset of all incidences determined by $e$ with respect to the set of all incidences in  $\mathcal{H}$, i.e., $j_e$ is an $mr$-dimensional $(0,1)$-vector whose $(w, f)$-th entry is one if and only if $w\in f=e$.

\begin{lemma}\label{lem_basis_kernelspace}
 The kernel of the generalized incidence matrix $B$ of an $r$-tree $\mathcal{T}$ with edge set $\{e_1,  e_2,\ldots , e_m\}$ has $ j_{e_1}, j_{e_2},\ldots ,j_{e_m} $  as a basis.
\end{lemma}

\begin{proof}
It is easy to verify that $B j_{e_1}=\cdots=B j_{e_m}=0$ and that $j_{e_1}, j_{e_2},\ldots ,j_{e_m}$ are linearly independent vectors. In addition, the dimension of kernel of $B$ is equal to the number of columns minus the rank of $B$, which equals $mr-n=m$, by Lemma~\ref{lem_rank} and Lemma~\ref{hypertree-order-size}.
\end{proof}

A configuration $\alpha$ on a hypergraph is said to be {\em recurrent} if the corresponding monomial  $x^{\alpha}$ can be toppled resulting in the same monomial after a nonempty sequence of topplings. Thus if configuration $\alpha$ is recurrent and has $z$ as its representative vector, then $\alpha=\alpha-Bz$ and so $Bz=0$. By Lemma~\ref{lem_basis_kernelspace}, $z$ can be uniquely expressed as a linear combination of $j_{e_1},  \ldots ,j_{e_m}$, with nonnegative integral  coefficients. This leads to an important property of recurrent configuration as follows.

\begin{lemma} \label{lem_sametimestoppled}
 Suppose $\alpha$ is a recurrent configuration and $e$ is an edge of  hypertree $\mathcal{T}$. Then each of  incidences   determined by $e$ must be toppled the same (maybe zero) times in the toppling sequence from $\alpha$ to itself.
\end{lemma}

\begin{proof}
Let $z$ be the representative vector of the toppling sequence from  $\alpha$ to itself. Then $Bz=0$ and so  $z=c_1j_{e_1}+c_2j_{e_2}+\cdots +c_mj_{e_m}$  by Lemma~\ref{lem_basis_kernelspace}, and the coefficients $c_1, \ldots, c_m$ must be nonnegative integers because $z$ is a nonnegative integral vector and $j_{e_1}, j_{e_2},\ldots ,j_{e_m}$ are disjoint.
\end{proof}

As an immediate consequence, we have the following result.

\begin{prop} \label{coro_lengthmupltiple}
Any directed cycle of a toppled digraph of an  $r$-tree   has length equal to a multiple of $r$.
\end{prop}

\begin{figure}[!hbpt]
\begin{center}

\includegraphics[scale=0.9]{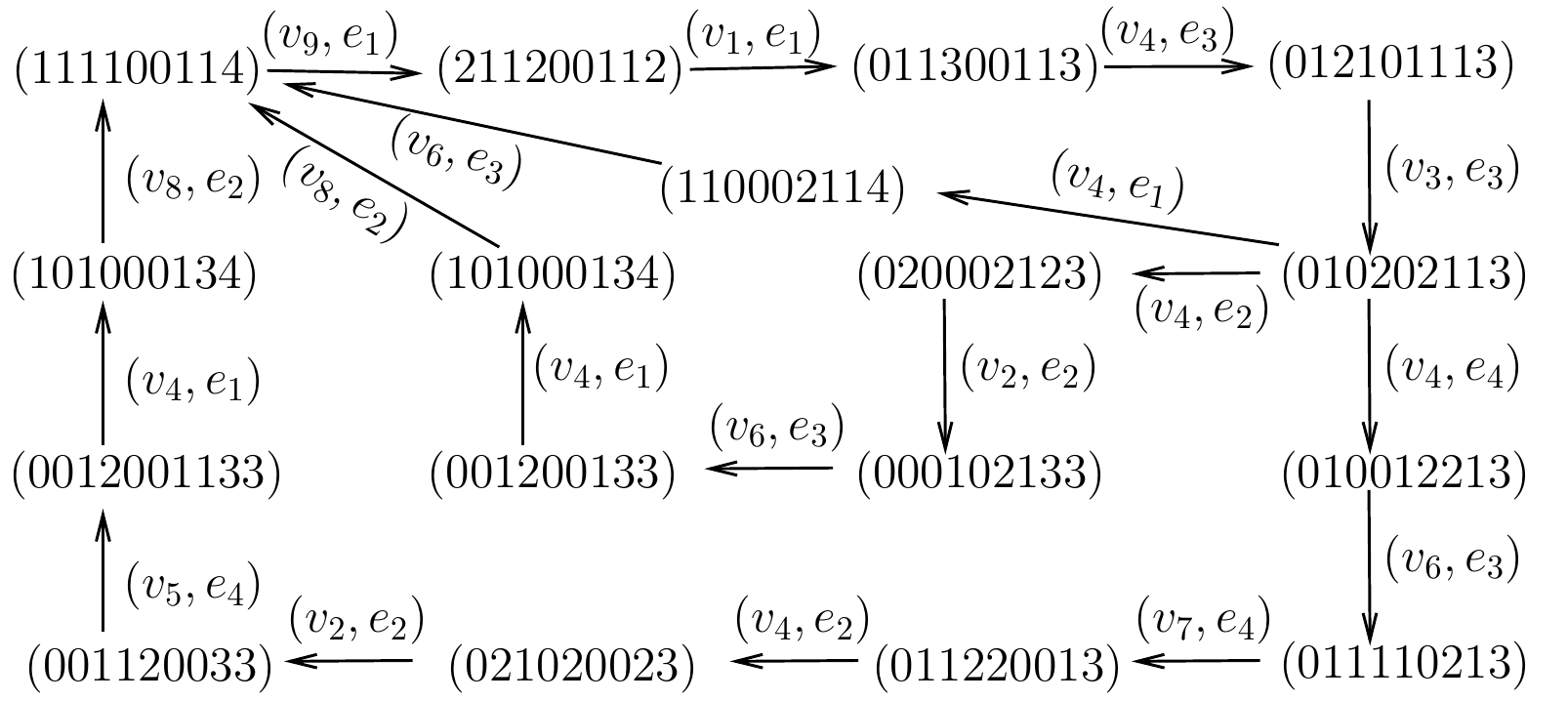}
\includegraphics[scale=0.9]{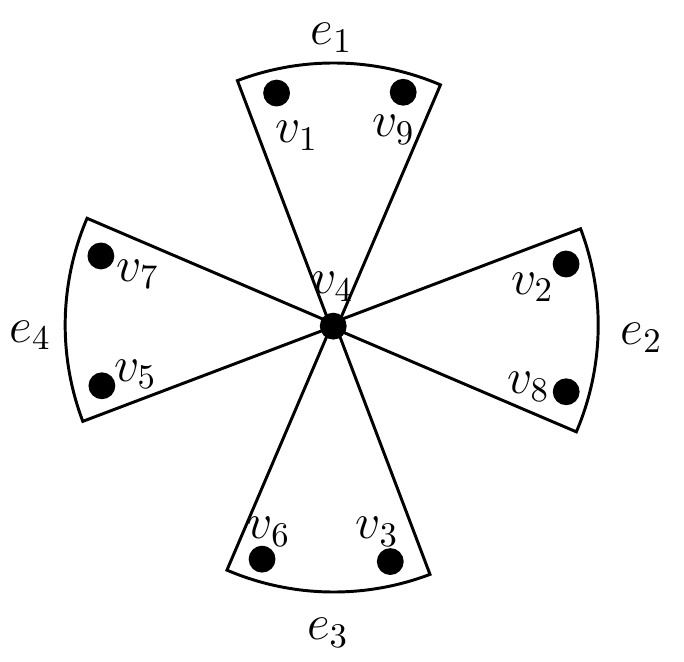}
\includegraphics[scale=0.9]{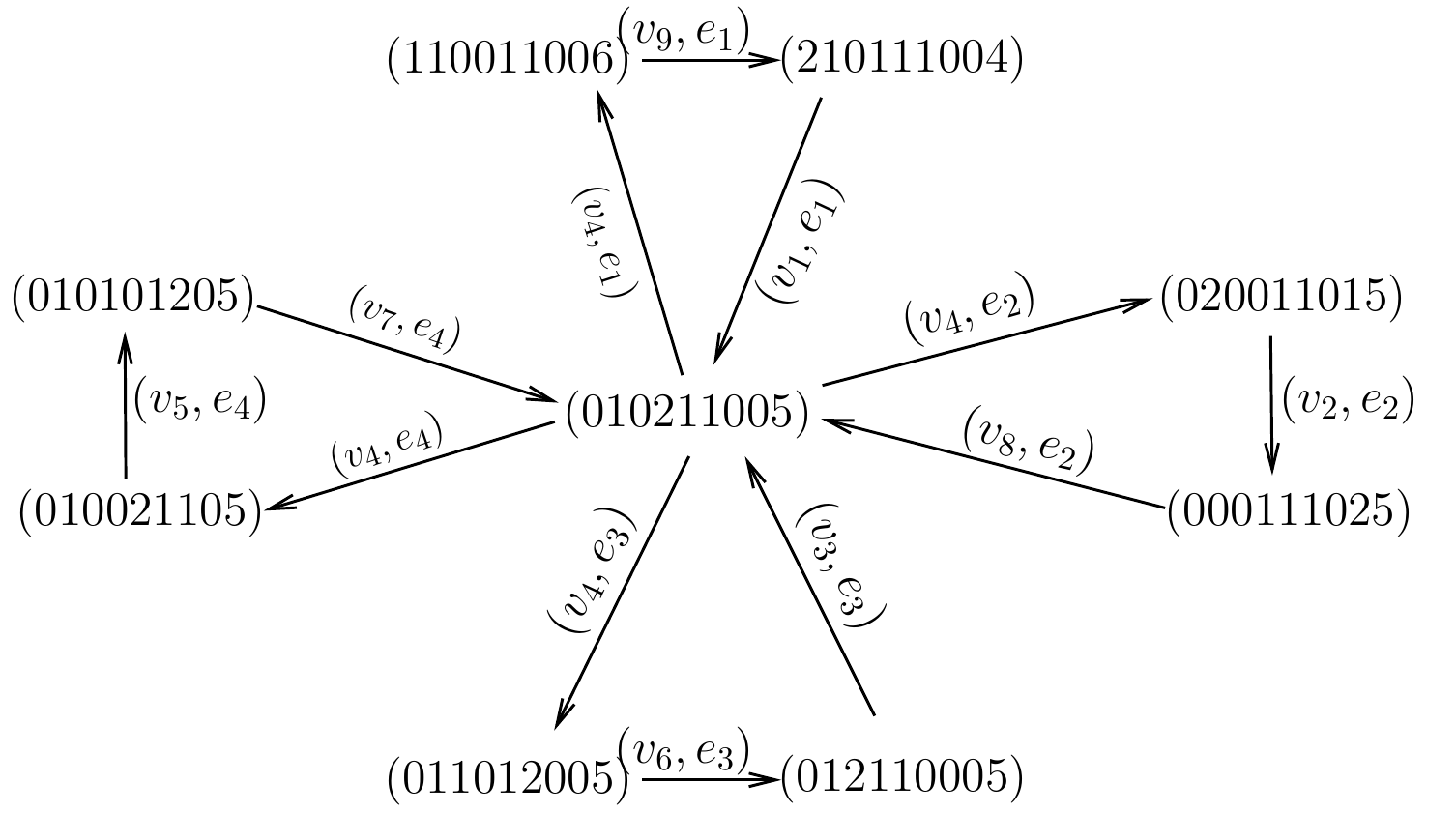}

\caption{Some cycles of length $3k$ for any $1\leq k\leq 4$ in the toppled digraph of a 3-tree  with an ordering given by  $x_i>x_j$ whenever $i<j$. }
\end{center}\label{fig_6cycle}
\end{figure}

\begin{example}\label{exam1}
For a hypergraph, its toppled digraph has a very large number of vertices and can be complicated in structure. However, this is the price  paid to reduce the characteristic polynomial of a hypergraph (defined by the resultant and difficult to compute) to the characteristic polynomial of a digraph. Take the hypertree in Fig.~1 as an example. It has $v_9$ as the root and the toppled digraph $\mathfrak{D}_9$ has $\binom{18}{8}=43758$ vertices. It is difficult to check the entire digraph. For our specific purpose, we just  consider some cycles to  provide some insights on the overall cycle structure of the toppled digraph. Observe that some directed cycles in a toppled digraph    may involve more than one edge from the corresponding toppling sequence.
\end{example}

It seems interesting to investigate the cycle structure of the toppled digraphs associated to a given $r$-tree $\mathcal{T}$, and we propose the following related query.

\begin{conj}
Let $\mathcal{T}$ be an $r$-tree with maximum degree $\Delta$ and with $r\geq 3$. Then some toppled digraph of $\mathcal{T}$ contains cycles of length $kr$, for any $k=1, \ldots, \Delta$.
\end{conj}

From Lemma~\ref{lem_sametimestoppled}, we observe that long cycles (i.e. length of $kr$ for $k>1$) of a toppled digraph of an $r$-tree must involve several edges. From Example~\ref{exam1},  it appears that long cycles involve toppling in which  hops (via a common vertex) occur among different edges. This observation serves as part of our intuition in forming the above conjecture.

\section{Main results}

For  a graph (or hypergraph) $H=(V(H), E(H))$, we simply write $v\in H$ instead of $v\in V(H)$ to mean that $v$ is a vertex of $H$, and use $|H|:=|V(H)|$ and $e(H):=|E(H)|$ to denote the order and size of   $H$, respectively.

Let  $\mathcal{H}_1$ and  $\mathcal{H}_2$  be two $r$-graphs with $v_1\in  \mathcal{H}_1$ and $v_2\in \mathcal{H}_2$.  We use $\mathcal{H}_1\circ\mathcal{H}_2$ to denote the $r$-graph obtained by identifying $v_1$ and $v_2$, called the \textit{coalescence} of $\mathcal{H}_1$ and  $\mathcal{H}_2$. The identified vertex in $\mathcal{H}_1\circ\mathcal{H}_2$ is called the \textit{coalesced vertex}. Note that $E(\mathcal{H}_1\circ\mathcal{H}_2)=E(\mathcal{H}_1)\cup E(\mathcal{H}_2)$ and $V(\mathcal{H}_1\circ\mathcal{H}_2)=(V(\mathcal{H}_1)\setminus\{v_1\})\cup (V(\mathcal{H}_2)\setminus\{v_2\}\cup \{\mathrm{the\,\, coalesced \,\, vertex}\}$.

\begin{lemma}\label{cutvertex}
Let  $\mathcal{H}$   be an $r$-graph   which is the coalescence of two $r$-graphs $\mathcal{H}_1$ and  $\mathcal{H}_2$,   with $v$ as the coalesced vertex. Let $\alpha$ and $\beta$ be two configurations on $\mathcal{H}$ such that $\alpha(u)<r-1$ for any $u\in \mathcal{H}_2\setminus\{v\}$ and $\beta(u)\geq r-1$ for some $u\in \mathcal{H}_2\setminus\{v\}$. If $\beta$ is reachable from $\alpha$, then in the toppling process from $\alpha$ to $\beta$ no vertex of $\mathcal{H}_2\setminus\{v\}$ can be toppled before $v$ is toppled inside some edge of $\mathcal{H}_2$.
\end{lemma}
\begin{proof}
Assume that the toppling sequence from  $\alpha$ to $\beta$ is
\begin{equation*}
  \beta_0 \stackrel{(i_1,e_1)}{\longrightarrow}  \beta_1\stackrel{(i_2,e_2)}{\longrightarrow} \beta_2\rightarrow\cdots     \stackrel{(i_{k-1},e_{k-1})}{\longrightarrow}   \beta_{k-1}\stackrel{(i_{k},e_{k})}{\longrightarrow} \beta_k,
 \end{equation*}
 where $\beta_0=\alpha, \beta_k=\beta$.

 Without loss of generality, suppose that $\beta_j(a)<r-1$ for any $a\in \mathcal{H}_2\setminus\{v\}$ and for any $j=0, 1, \ldots, k-1$, and $\beta_k(u)\geq r-1$ for some $u\in \mathcal{H}_2\setminus\{v\}$. Let $z$ be the representative vector of the toppling sequence $\{(i_1,e_1), (i_2,e_2), \ldots,     (i_{k},e_{k})\}$.
 Then
 \begin{align*}
 \beta(u)=\alpha(u)-(r-1)\sum_{e\in E_u}z(u,e)+\sum_{f\in E_u, u\neq w\in f}z(w, f).
 \end{align*}
 Since $\beta(u)\geq r-1>\alpha(u)$, there must exist an incidence $(w, f)$ satisfying $z(w, f)>0$, where $f\in E_u$ and $u\neq w\in f$.
 By our assumption on $\beta_0,  \beta_1, \ldots,   \beta_{k-1}$, we know that all vertices of $\mathcal{H}_2$ other than $v$ are stable with respect to $\beta_0,  \beta_1, \ldots,   \beta_{k-1}$ and so  there must exist an incidence $(v, f)$ satisfying  $\{v, u\}\subseteq f$ and $z(v, f)>0$, which completes the proof.
\end{proof}

Let $\mathcal{T}$ be an $r$-tree of order  $n+1$. Choose any given vertex $v$ of $\mathcal{T}$ as the root,
  and relabel it as 0. Consider the distance partition from the root in  $\mathcal{T}$, and relabel the  remaining vertices as $1,\ldots, n$ according to the following rule:
\begin{itemize}
  \item  $\mathrm{dist}(0, i)<\mathrm{dist}(0, j)\Rightarrow i<j$;
   \item for each edge $e=\{i_1, \ldots, i_r\}$, there exists only one vertex (called the root of   $e$), say $i_1$, of distance from the root 0 one less  than the other vertices in $e$, then $i_2, \ldots, i_r$ is a set of consecutive integers.
\end{itemize}
We adopt the ordering on vertices such that the variables can be ordered as  $x_n>x_{n-1}>\cdots>x_0$, and  call such  ordering a \textit{good ordering} of variables (or vertices) for hypertree $\mathcal{T}$. Then we use $\mathfrak{D}$ denote the toppled digraph with designated root understood from context. Moving forward, we use  $S_i$ to denote the set
\[
 S_i=\{x^{\alpha}: |\alpha|=d,\,\, x_i^{r-1}    \,\, \mathrm{divides}\,\, x^{\alpha}, \,\, \textrm{but}\,\, x_j^{r-1}\,\,\mathrm{does\,\, not}, \forall j>i\},
\]
for $i=0, 1, \ldots, n$. Note that this definition coincides with \eqref{Si} although the ordering has changed.

 For an edge $e$ of $\mathcal{T}$, the \textit{level} of $e$  is defined to be the distance from the root of $e$ to the root of $\mathcal{T}$ and is denoted by $\mathfrak{L}(e)$. For a vertex $v$ and an edge $e$, $e$ is said to be  an \textit{upper edge} or a \textit{lower edge}  for $v$ according to whether the level of $e$ less than   $\mathrm{dist}(0, v)$ or not, respectively.

Let $\mathcal{T}$ be an $r$-tree on the vertex set $[n]_0$  with a good ordering. For each edge $e$ with its root $v$, the incidence $(v, e)$ is called   the \textit{principal incidence determined by} $e$ and every  $(w, e)$ with $w\neq v$  a \textit{non-principal incidence  determined by} $e$. Let $\mathcal{Z}$ be any  toppling sequence in which a principal incidence     $(v, e)$ occurs.  If some non-principal incidence  determined by $e$ occurs in $\mathcal{Z}$ after this  $(v, e)$, then there is a unique one, say $(u, e)$, such that the subsequence of $\mathcal{Z}$  from $(v, e)$ to $(u, e)$, denoted by $\mathcal{Z}_1$,  satisfies that
\begin{itemize}
  \item after $\mathcal{Z}_1=\{(v, e), \ldots, (u, e)\}$ in $\mathcal{Z}$, either no incidence determined by $e$ occurs,  or an incidence  determined by $e$ does occur and  the first such  incidence  is another  $(v, e)$;
  \item except for the first incidence $(v, e)$, no other $(v, e)$ occurs  in $\mathcal{Z}_1$.
\end{itemize}
Otherwise, set $\mathcal{Z}_1=\{(v, e)\}$.  Such   a subsequence $\mathcal{Z}_1$, denoted by $\mathcal{Z}(v ,e)$,  is called a \textit{proper toppling sequence} of $(v, e)$ in $\mathcal{Z}$.

Let $\mathcal{T}$ be an $r$-tree on the vertex set $[n]_0$  with a good ordering. For each edge $e$ with its root $v$,  let $\mathcal{T}_e$ be the component of $\mathcal{T}-(E_v\setminus\{e\})$ containing the edge $e$, and let $\mathcal{T}_e'$ be the component of $\mathcal{T}-e$ containing vertex  $v$.

\begin{coro} \label{prinfirst}
Let $\mathcal{Z}$ be any given   toppling sequence on $\mathcal{T}$ starting from a principal incidence $(v, e)$. If an incidence $(u, f)\in \mathcal{Z}$ and the level of $e$ is no more than that of $f$,   and no other incidence determined by $f$ occurs before $(u, f)$  in $\mathcal{Z}$,  then $(u, f)$ must be a principal incidence.
\end{coro}
\begin{proof}
 Assume that $w$ is the root of $f$.    Observe then that $\mathcal{T}$ is the coalescence of $\mathcal{T}_f$ and $\mathcal{T}_f'$, and the coalesced vertex is $w$. Furthermore, $f$ is an edge of $\mathcal{T}_f$  whose deletion from $\mathcal{T}$ results in $\mathcal{T}_f'$.

 Assume that $\alpha$ and $\beta$ are the configurations that $(v, e)$ and $(u, f)$ follow in $\mathcal{Z}$, respectively.
 Since the level of $e$ is no more than that of $f$, $v< a$ for any $a\in \mathcal{T}_f\setminus\{w\}$, together with $\alpha\in S_v$ implies that
\begin{equation}\label{alphaTf}
\alpha(a)<r-1, \,\,  \forall  a\in \mathcal{T}_f\setminus\{w\}.
\end{equation}

   Suppose by contradiction that $u\neq w$. Noting $\beta(u)\geq r-1$ and using \eqref{alphaTf}, by Lemma~\ref{cutvertex}, in the toppling sequence $\mathcal{Z}$ before $(u, f)$, $w$ must be toppled inside some edge of  $\mathcal{T}_f$, which must be $f$ because it is the only edge of $\mathcal{T}_f$ incident to $w$. This contradicts the assumption that no other incidence determined by $f$ occurs before $(u, f)$  in $\mathcal{Z}$.
\end{proof}

For a proper toppling sequence $\mathcal{Z}(v ,e)$, let $\widehat{\mathcal{Z}}(v ,e)$ be the set of all principal incidences determined by an edge of  $\mathcal{T}_e$   in $\mathcal{Z}(v ,e)$.
 The \textit{depth} of a proper toppling sequence $\mathcal{Z}(v ,e)$  is defined to be the maximum distance in $\mathcal{T}$ between $v$ and   $w$ with $(w, f)\in \widehat{\mathcal{Z}}(v ,e)$ for some edge $f\in E(\mathcal{T}_e)$.

\begin{lemma}\label{lem_properseq}
Let $\mathcal{T}$ be an $r$-tree on the vertex set $[n]_0$ with a good ordering, and $\mathcal{Z}$ any toppling sequence on $\mathcal{T}$. For any  principal incidence $(v, e)\in\mathcal{Z}$, let  $\mathcal{Z}(v ,e)$ be any fixed proper toppling sequence of $(v, e)$ in   $\mathcal{Z}$. If $\widehat{\mathcal{Z}}(v ,e)\setminus\{(v ,e)\}\neq\emptyset$, then for any  $(u, f)\in \widehat{\mathcal{Z}}(v ,e)\setminus\{(v ,e)\}$,     the depth of  $\mathcal{Z}(v ,e)$  is greater than that of  $\mathcal{Z}(u ,f)$.
\end{lemma}

\begin{proof}
For convenience, let $\mathcal{Z}_1=\mathcal{Z}(v ,e)$ and assume that
\[
\mathcal{Z}_1=\{(v, e), \ldots, (u, f), \ldots  (w, e)\}.
\]
 Let $\mathcal{Z}_1(u, f)$ be the proper toppling sequence of $(u, f)$ in   $\mathcal{Z}_1$.
 First we show that $\mathcal{Z}(u, f)=\mathcal{Z}_1(u, f)$. Since $\mathcal{Z}(u, f)\supseteq\mathcal{Z}_1(u, f)$ holds trivially, it suffices to show that $\mathcal{Z}(u, f)\subseteq\mathcal{Z}_1(u, f)$.
By contradiction, suppose that $\mathcal{Z}(u, f)\nsubseteq \mathcal{Z}_1(u, f)$ and
\[
\mathcal{Z}(u, f)=\{(u, f), \ldots  (w, e), \ldots,    (u', f)\}.
\]
Let $\mathcal{Z}_2=\mathcal{Z}(u, f)$. Observe that the level of $e$ is less than that of $f$ by the form of $\widehat{\mathcal{Z}}(v ,e)$, and   no $(u, f)$ occurs between $(w, e)$  and  $(u', f)$ in $\mathcal{Z}_2$. This contradicts the conclusion of Corollary~\ref{prinfirst}.   Thus $\mathcal{Z}(u, f)=\mathcal{Z}_1(u, f)$. Consequently, the depth of  $\mathcal{Z}(v ,e)$  is greater than that of  $\mathcal{Z}(u ,f)$.
\end{proof}

\begin{lemma}\label{lem_toppletimes}
Let $\mathcal{T}$ be an $r$-tree on the vertex set $[n]_0$ with a good ordering. In any toppling sequence $\mathcal{Z}$ on $\mathcal{T}$, for any  principal incidence $(v, e)\in\mathcal{Z}$,  each non-principal   incidence  determined by  $e$ occurs at most once  in any proper toppling sequence of $(v, e)$ in $\mathcal{Z}$.
\end{lemma}

\begin{proof}
For any given principal incidence $(v, e)$  in $\mathcal{Z}$, let $\alpha$ be the configuration   $(v, e)$ follows in $\mathcal{Z}$. Then $\alpha(w)<r-1$ for any $w>v$. Let $L$ be the depth of this proper toppling sequence of $(v, e)$ in   $\mathcal{Z}$.

By induction on the depth $L$,  we show that for any principal incidence $(v, e)$ in $\mathcal{Z}$,  each non-principal  incidence  determined by  $e$   occurs at most once  in any proper toppling sequence of $(v, e)$ in $\mathcal{Z}$. If $L=0$, then each non-root vertex of $e$ cannot topple in a lower edge and can only topple inside $e$ following this proper toppling sequence.
To reach a contradiction, suppose that some non-principal  incidence, say $(u, e)$, occurs at least twice in a proper toppling sequence    $\mathcal{Z}(v, e)$. In particular, assume that
\begin{equation}\label{twicenonprin}
\mathcal{Z}(v, e)=\{(v, e),\ldots, \underbrace{(u, e)}_{ \textcircled{1}}, \ldots, (*, *),  \underbrace{(u, e)}_{\textcircled{2}}, \ldots, (v', e)\},
\end{equation}
and $\textcircled{2}$ is the first of the non-principal   incidences determined by $e$ which is a repeat in \eqref{twicenonprin}. Let $\beta$ be the configuration that the incidence $\textcircled{2}$ follows in \eqref{twicenonprin}.   Let $z$ be the representative vector of the subsequence $\{(v, e),\ldots, \underbrace{(u, e)}_{ \textcircled{1}}, \ldots,  (*, *)\}$ of $\mathcal{Z}(v, e)$ in \eqref{twicenonprin}. Then
\begin{align}\label{nonprin}
  \beta(u)&=\alpha(u)-(r-1)\sum_{f\in E_u}z(u, f)+\sum_{f\in E_u, u\neq w\in f}z(w, f).
\end{align}

Note that   $\sum_{f\in E_u}z(u, f)=1$ as $z(u, e)=1$ and  $z(u, f)=0$ for any $e\neq f\in E_u$.

For any $e\neq f\in E_u$,     $u$  is the root of $f$ and then $z(u, f)=0$ implies that $z(w, f)=0$ for any $u\neq w\in f$ by Corollary~\ref{prinfirst}. By assumption on the incidence  $\textcircled{2}$, we know that all incidences determined by $e$ in $\mathcal{Z}(v, e)$  before $\textcircled{2}$ are distinct and so $z(w, e)\leq 1$ for any $u\neq w\in e$.  Thus the third term of the right side in \eqref{nonprin} is equal to $\sum_{u\neq w\in e}z(w, e)\leq r-1$.

Therefore $\beta(u)=\alpha(u)-(r-1)+\sum_{u\neq w\in e}z(w, e)\leq \alpha(u)$. However,   $r-1\leq\beta(u)$ and $\alpha(u)<r-1$,  a contradiction.

Suppose that the conclusion holds for any  principal incidence $(v, e)$  in any proper toppling sequence of $(v, e)$ in $\mathcal{Z}$ with its depth less than  $L$. For the sake of a contradiction, assume that for a principal incidence $(v, e)$, there is a proper toppling sequence $\mathcal{Z}_1$  of $(v, e)$ in $\mathcal{Z}$ with depth $L$ such that some non-principal incidence, say $(u ,e)$,  occurs at least twice in $\mathcal{Z}_1$. That is, suppose
 \begin{equation}\label{nonprincipal}
\mathcal{Z}_1=\{(v, e),\ldots, \underbrace{(u, e)}_{ \textcircled{1}}, \ldots, (*, *),  \underbrace{(u, e)}_{\textcircled{2}}, \ldots, (v', e)\}
\end{equation}
  and $\textcircled{2}$ is the first of the non-principal   incidences determined by $e$ which is repeated in $\mathcal{Z}_1$.
  Assume that  $\beta$ is the configuration that the incidence $\textcircled{2}$  follows in $\mathcal{Z}_1$ in (\ref{nonprincipal}), and $z$ is the representative vector of the toppling  subsequence  $\{(v, e),\ldots, \underbrace{(u, e)}_{ \textcircled{1}}, \ldots, (*, *)\}$ in \eqref{nonprincipal}.

   If  $\deg(u)=1$ or $u$ is only toppled inside $e$ following this proper toppling sequence, then applying a similar argument as used to verify the base case for induction,  we have $r-1\leq \beta(u)=\alpha(u)-(r-1)z(u, e)+\sum_{u\neq w\in e}z(w, e) \leq \alpha(u)<r-1$, a contradiction.

Now suppose that $\deg(u)>1$ and $u$   is toppled in a lower edge for $u$  following this proper toppling sequence.  For any principal incidence $(u, f)\in \mathcal{Z}_1$, the depth of  $\mathcal{Z}(u, f)$  is less than $L$ by Lemma~\ref{lem_properseq}. By the induction hypothesis, each non-root of $f$ is toppled at most once in any proper toppling sequence of $(u, f)$ in $\mathcal{Z}$.
  For any $e\neq f\in E_u$, $u$ is the root of $f$ and so additionally  by Corollary~\ref{prinfirst}, we have $z(w, f)\leq z(u, f)$   for any $u\neq w\in f$.
Then
\begin{align*}
\beta(u)&=\alpha(u)-(r-1)\sum_{f\in E_u}z(u, f)+\sum_{f\in E_u, u\neq w\in f}z(w, f)\\
&=\alpha(u)-(r-1)+\sum_{u\neq w\in e}z(w, e)-\sum_{e\neq f\in E_u}((r-1)z(u, f)-\sum_{u\neq  w\in f}z(w, f))\\
&\leq \alpha(u),
\end{align*}
 where the final inequality follows from the fact that for any $u\neq w\in e$,   $z(w, e)\leq1$, and
  for every $e\neq f\in E_u$,
  \begin{align*}
  (r-1)z(u, f)-\sum_{u\neq  w\in f}z(w, f)\geq (r-1)z(u, f)-\sum_{u\neq  w\in f}z(u, f)\geq0.
  \end{align*}  Thus $r-1\leq \beta(u)  \leq \alpha(u)<r-1$, a contradiction.
\end{proof}

\begin{lemma}\label{lem_oncetopple}
Let $\mathcal{T}$ be an $r$-tree on the vertex set $[n]_0$ with a good ordering, and $\alpha$ a recurrent configuration on $\mathcal{T}$. In any toppling sequence $\mathcal{Z}$  from $\alpha$ to itself, viewed cyclically, for any $(v, e)\in \mathcal{Z}$,     each incidence determined by $e$   occurs exactly once between $(v, e)$ and before the   next occurrence of $(v, e)$  in $\mathcal{Z}$.
\end{lemma}
\begin{proof}
Since $(v, e)$ is toppled from $\alpha$ to itself,   by Lemma~\ref{lem_sametimestoppled},  every incidence determined by $e$ and particularly the principal incidence,    say $(u, e)$,  must occur in $\mathcal{Z}$. By Lemma~\ref{lem_toppletimes}, each incidence determined by $e$ occurs at most once in any proper toppling sequence of $(u, e)$ in $\mathcal{Z}$ and by Lemma~\ref{lem_sametimestoppled} meantime, it must occur exactly once in any proper toppling sequence of $(u, e)$ in $\mathcal{Z}$. Hence for any incidence $(v, e)$, every incidence determined by $e$ occurs once in $\mathcal{Z}$   between $(v, e)$ and    before the next occurrence of $(v, e)$.
\end{proof}

Let $\mathcal{T}$ be an $r$-tree on the vertex set $[n]_0$  with a good ordering. For any edge $e$ of $\mathcal{T}$,
let $\hat{e}$ denote  the $(r-1)$-subset of vertices obtained from $e$ by deleting its root. Let %$j$
$\one$ be the all-ones vector of   dimension $r-1$.
A configuration $\alpha$ on $\mathcal{T}$ is said to be \textit{critical} if for each $e\in E(\mathcal{T})$,  $((r-1)\one-\alpha)|_{\hat{e}}$ is a parking function of length $r-1$.

For an edge $e=\{i_1, i_2,\ldots, i_r\}$ of   $\mathcal{T}$,   use $c_e$ denote a directed $r$-cycle $i_1 \rightarrow i_2\rightarrow \cdots\rightarrow i_r\rightarrow i_1$ (given $e$, there are a number of arcs that can form a cycle but we select just one). Let $\mathcal{\vec{T}}$ denote a directed graph with the same vertex set $[n]_0$ which has $E(c_{e_1})\cup E(c_{e_2})\cup\cdots \cup E(c_{e_m})$ as arc set, and call it a \textit{representative digraph} of $\mathcal{T}$. Observe that for any two edges $e_i, e_j$ of  $\mathcal{T}$, we have  $e_i\cap e_j=c_{e_i}\cap c_{e_j}$. For example, the bottom digraph  in Figure~1 is isomorphic to a representative digraph of the hyperstar above it.

We now state and prove a key lemma needed to establish our main observation. First we define the concept of a hop with regard to a toppling sequence. Given a toppling sequence $\mathcal{Z}$ on a hypertree $\mathcal{T}$, if there exist two consecutive incidences in  $\mathcal{Z}$, say
\[
\mathcal{Z}=\{\ldots, (i,e), (j,f), \ldots\}
\]
 such that $j \not\in e$, then we say a {\em hop} occurs in this sequence  $\mathcal{Z}$.

\begin{lemma}\label{lem_representative-digraph}
Let $\mathcal{T}$ be an $r$-tree on the vertex set $[n]_0$  with a good ordering.
Then the  strong component  of the toppled digraph  to which a critical configuration  belongs is isomorphic to a representative digraph of $\mathcal{T}$.
\end{lemma}

\begin{proof}
Let $\mathfrak{D}$ denote the topple digraph of $\mathcal{T}$, and $\alpha$   a  critical configuration on $\mathcal{T}$. Let $C_{\alpha}$ denote the strong component of $\mathfrak{D}$ that contains   $\alpha$.
%In a toppling sequence, if   a vertex   is to be toppled   and  is not incident to the edge inside which the  very previous  vertex  just toppled, then we say a {\em jump} occurs.
The remainder of the proof is divided into a series of claims proved in succession.

\textbf{Claim 1.} For any     $\beta\in \mathfrak{D}$ which is reachable from $\alpha$, if the vertices   toppled in a toppling sequence from $\alpha$ to $\beta$ are distinct, then no hopping occurs. Furthermore,  if we assume that  $\beta\in S_{w}$, then either $w$ is different from any vertex toppled in the toppling sequence from $\alpha$ to $\beta$, or $w$ is the    root of the last edge toppled.

\begin{proof}
Assume that $\beta\neq\alpha$ and the toppling sequence from $\alpha$ to $\beta$ is given by
\begin{equation}\label{topplesequence}
\{(i_1^{(1)}, e_1), \ldots, (i_{s_1}^{(1)}, e_1), (i_1^{(2)}, e_{2}), \ldots, (i_{s_2}^{(2)}, e_{2}), \ldots, (i_1^{(k)}, e_{k}), \ldots, (i_{s_k}^{(k)}, e_{k})\},
\end{equation}
where $k\geq1$ and when $k\geq2$, $e_{j}\neq e_{j+1}$ for any $j=1, \ldots, k-1$.
In this toppling sequence the resulting configuration  is denoted by $\alpha_j$ after toppling the incidence $(i_{s_j}^{(j)}, e_{j})$ immediately, for $j=1, \ldots, k$. Note that $i_1^{(1)}$ is the root 0, and  $\alpha_k=\beta$. Assume that $\beta\in S_{w}$. Since $\alpha_1$ is obtained from the critical configuration $\alpha$ by toppling some incidences determined by $e_1$, $\alpha_1(v)\geq r-1$ implies that $v\in e_1$. However,   $\alpha_1\in S_{i_1^{(2)}}$ and so $i_1^{(2)}\in e_1$.

By contradiction,   suppose some vertex $i_1^{(l+1)}$ with $1< l\leq k$ such that  $i_1^{(l+1)}\notin e_{l}$ where, by convention, we assume that $i_1^{(k+1)}=w$. Further, we assume $i_1^{(l+1)}$ is the first of all vertices toppled satisfying this hypothesis, that is,
\[
i_1^{(2)}\in e_{1}, i_1^{(3)}\in e_{2}, \ldots, i_1^{(l)}\in e_{l-1},\quad  \mathrm{but}\quad i_1^{(l+1)}\notin e_{l}.
\]
Observe that for any $u\in e_l\backslash\{i_1^{(l)}\}$,  $i_1^{(1)}  e_{1}  i_1^{(2)}  e_{2}\cdots i_1^{(l)}  e_{l}u$ is a walk. Since the vertices  $i_1^{(1)},  i_1^{(2)}, \ldots,  i_1^{(l)}$ are distinct and $u\neq i_1^{(l)}$, $e_{j}\neq e_{j+1}$ for  $j=1, \ldots, k-1$,
 if either $u\in \{i_1^{(1)},  i_1^{(2)}, \ldots,  i_1^{(l-1)}\}$ or   $e_{1},    e_{2}, \ldots,   e_{l}$  are not distinct, then  the walk $i_1^{(1)}  e_{1}  i_1^{(2)}  e_{2}\cdots i_1^{(l)}  e_{l}u$ will contain a cycle of $\mathcal{T}$,  which is not possible. Thus $i_1^{(1)}  e_{1}  i_1^{(2)}  e_{2}\cdots i_1^{(l)}  e_{l}u$ is a path in $\mathcal{T}$.

Since $\alpha_l$ is obtained from the critical configuration $\alpha$ by toppling some incidences determined by $e_1, \ldots, e_{l}$,
  $\alpha_l(v)\geq r-1$ implies that  $v\in e_1\cup \cdots\cup e_{l}$. Then as   $i_1^{(l+1)}\notin e_l$ and $\alpha_l\in S_{i_1^{(l+1)}}$ it follows that $i_1^{(l+1)}\in(e_1\cup \cdots\cup e_{l-1})\setminus e_l$.

Furthermore, since   $\alpha_{l-1}$ is obtained from   $\alpha$ by toppling some incidences determined by $e_1, \ldots, e_{l-1}$,
  $\alpha_{l-1}(v) \neq \alpha(v)$ implies that $v\in e_1\cup \cdots\cup e_{l-1}$ and so  $\alpha_{l-1}|_{\hat{e}_{l}}=\alpha|_{\hat{e}_{l}}$.
  Then $((r-1)\one-\alpha_{l-1})|_{\hat{e}_{l}}$ is a parking function and  there must exist a vertex $v\in e_{l}$  such that $\alpha_l(v)\geq r-1$.

If  $v\neq i_1^{(l)}$, then $v$ is a non-root vertex in $e_{l}$. Since $i_1^{(1)}  e_{1}  i_1^{(2)}  e_{2}\cdots i_1^{(l)}  e_{l}v$ is path from the root 0 to $v$, $\mathrm{dist}(0, i_1^{(l+1)})<\mathrm{dist}(0, v)$ as  $i_1^{(l+1)}\in(e_1\cup \cdots\cup e_{l-1})\setminus e_l$. Thus  $x_v>x_{i_1^{(l+1)}}$, a contradiction with that $\alpha_{l}\in S_{i_1^{(l+1)}}$.

  If   $v=i_1^{(l)}$, then all vertices in $e_l$ have been toppled (exactly once) in the toppling sequence from $\alpha$ to $\alpha_l$ because $((r-1)\one-\alpha_{l-1})|_{\hat{e}_{l}}$ is a parking function, and so
   $\alpha_{l}=\alpha_{l-1}$. However,   $\alpha_{l-1}=\alpha_{l}\in S_{i_1^{(l+1)}}$ and $\alpha_{l-1}\in S_{i_1^{(l)}}$. This leads to  a contradiction since $i_1^{(l+1)}\neq i_1^{(l)}$ if $l<k$.
The remaining case is $l=k$ and $w=i_1^{(k+1)}= i_1^{(k)}$. Hence $w\in e_{k}$.

Now we prove the second statement in the claim. If $w\notin\{i_1^{(1)}, \ldots  i_{s_1}^{(1)}, i_1^{(2)},  \ldots, i_{s_2}^{(2)}, \ldots, i_1^{(k)}, \ldots, i_{s_k}^{(k)}\}$, there is nothing to prove. Otherwise, $w\in e_{k}$ from above, and $w\in \{i_1^{(k)}, \ldots, i_{s_k}^{(k)}\}$. Since  $((r-1)\one-\alpha_{k-1})|_{\hat{e}_{k}}$ is a parking function, if $s_k<r$, then there must exist a vertex $u\in e_k\backslash\{i_1^{(k)},   \ldots, i_{s_k}^{(k)}\}$ such that   $\beta\in S_u$. This contradicts the assumption that $\beta\in S_w$ and $w\in \{i_1^{(k)}, \ldots, i_{s_k}^{(k)}\}$. Therefore $s_k=r$ and this implies that $w=i_1^{(k)}$.
\end{proof}

For the given toppling sequence  from $\alpha$ to $\beta$ in \eqref{topplesequence},  the abbreviated sequence $i_1^{(1)}  e_{1}  i_1^{(2)}  e_{2}\cdots i_1^{(k)}  e_{k}w$ is called a \textit{contraction} of the toppling sequence \eqref{topplesequence}.

\textbf{Claim 2.} For any two   $\beta, \gamma\in   \mathfrak{D}$ which both are reachable from $\alpha$, if there exists some $i$ ($0\leq i\leq n$) such that   $\beta, \gamma\in S_i$, then $\beta=\gamma$. Furthermore, the contraction of the toppling sequence corresponding to a shortest directed path from $\alpha$ to $\beta$ is the unique path from 0 to $i$ in $\mathcal{T}$.

\begin{proof}
Suppose, for the sake of a contradiction,  that $\beta \neq\gamma$. Then there are two distinct shortest directed paths in  $\mathfrak{D}$, say $P_1$ and $P_2$,  from  $\alpha$ to $\beta$ and $\gamma$, respectively. Assume the toppling sequences from  $\alpha$ to $\beta$  along $P_1$ and that from  $\alpha$ to $\gamma$ along $P_2$ are given by
$$\mathcal{Z}_1=\{(i_1^{(1)}, e_1), \ldots, (i_{r_1}^{(1)}, e_1), (i_1^{(2)}, e_{2}), \ldots, (i_{r_2}^{(2)}, e_{2}), \ldots, (i_1^{(k)}, e_{k}), \ldots, (i_{r_k}^{(k)}, e_{k})\},$$
 and
 \[
 \mathcal{Z}_2=\{(j_1^{(1)}, f_{1}), \ldots, (j_{s_1}^{(1)}, f_{1}), (j_1^{(2)}, f_{2}), \ldots, (j_{s_2}^{(2)}, f_{2}), \ldots, (j_1^{(l)}, f_{l}), \ldots, (j_{s_l}^{(l)}, f_{l})\},
 \]
 where $e_{\varsigma}\neq e_{\varsigma+1}$ for all $\varsigma=1, \ldots, k-1$ and  $f_{\varsigma}\neq f_{\varsigma+1}$ for all $\varsigma=1, \ldots, l-1$.
Use $\beta_{p}$ ($\gamma_{q}$, resp.) to denote the resulting configuration obtained by  toppling the incidence $ (i_{r_p}^{(p)}, e_{p})$ ($ (j_{s_q}^{(q)}, f_{q})$, resp.), for $p=1, \ldots, k$ ($q=1, \ldots, l$, resp.). Note that $\beta_{k}=\beta$ and $\gamma_{l}=\gamma$.

\textbf{Case 1.}  The collections $i_1^{(1)},  \ldots, i_{r_1}^{(1)},   \ldots, i_1^{(k)}, \ldots, i_{r_k}^{(k)}, i$ and $j_1^{(1)}, \ldots, j_{s_1}^{(1)}, \ldots, j_1^{(l)},   \ldots, j_{s_l}^{(l)}, i$ each consist of distinct integers.

By Claim 1, the contraction of $\mathcal{Z}_1$    is a  walk in $\mathcal{T}$ because  any two neighbouring elements are incident,  and is denoted by   $W_1=i_1^{(1)} e_{1}  i_1^{(2)}  e_{2}  \cdots   i_1^{(k)}  e_{k} i$.  Since $i_1^{(1)},    i_{1}^{(2)},   \ldots, i_1^{(k)},  i$  are distinct and   $e_{\varsigma}\neq e_{\varsigma+1}$ for  $\varsigma=1, \ldots, k-1$, the edges $e_{1},  e_{2},  \ldots,     e_{k}$ are distinct, otherwise $W_1$ contains a cycle of $\mathcal{T}$, which is not possible. Thus $W_1$ is a path in $\mathcal{T}$. Similarly it can be shown that   $W_2=j_1^{(1)} f_{1}  j_1^{(2)}  f_{2}  \cdots   j_1^{(l)}  f_{l} i$ is also a path in $\mathcal{T}$.  Note that $i_1^{(1)}=j_1^{(1)}=0$, which is the root of  $\mathcal{T}$. Thus  $W_1$ and  $W_2$ are two paths from 0 to $i$ in $\mathcal{T}$.  It follows that $W_1=W_2$ since the path in  $\mathcal{T}$  between any two vertices   is unique.   Hence $k=l$, $e_\varsigma=f_\varsigma$  and  $i_1^{(\varsigma)}=j_1^{(\varsigma)}$ for $\varsigma=1, \ldots, k$.

Since $\alpha$ is a critical configuration on $\mathcal{T}$ and $i_1^{(1)}=0$, the incidences determined by $e_1$ are toppled contiguously from the configuration $\alpha$. By Lemma~\ref{lem_unique-topple-order},    the toppling order  of such incidences  is determined uniquely  by $\alpha|_{\hat{e}_{1}}$, which together with the assumption of $i_1^{(2)}=j_1^{(2)}$,  implies that $(i_1^{(1)},  \ldots, i_{r_1}^{(1)})=(j_1^{(1)},  \ldots, j_{s_1}^{(1)})$, and so $\beta_1=\gamma_1$.  Since $\beta_1$ is obtained from $\alpha$ by toppling some incidences determined by $e_1$ and  $i_1^{(2)}\in e_1\cap e_2$ is the root of $e_2$, $\beta_1|_{\hat{e}_{2}}=\alpha|_{\hat{e}_{2}}$. It follows that  $((r-1)\one-\beta_1)|_{\hat{e}_{2}}=((r-1)\one-\alpha)|_{\hat{e}_{2}}$  is a parking function by the assumption on $\alpha$. Note that  any unstable vertex not in $e_2$ must be in $e_1$ so their labels are at most the labels of the vertices in $e_2$, which implies that the incidences determined by $e_2$ can be toppled contiguously from $\beta_1$.  Similarly it can be shown that $(i_1^{(2)},   \ldots, i_{r_2}^{(2)})=(j_1^{(2)},   \ldots, j_{s_2}^{(2)})$. Thus $\beta_2=\gamma_2$. Continuing in this way we have
$\beta_k=\gamma_l$, i.e. $\beta=\gamma$.

\textbf{Case 2.} Either collection $i_1^{(1)},  \ldots, i_{r_1}^{(1)},   \ldots, i_1^{(k)}, \ldots, i_{r_k}^{(k)}, i$ or $j_1^{(1)}, \ldots, j_{s_1}^{(1)}, \ldots, j_1^{(l)},   \ldots, j_{s_l}^{(l)}, i$ has repeated vertices.

Without loss of generality, assume that $i_1^{(1)},  \ldots, i_{r_1}^{(1)},   \ldots, i_1^{(k)}, \ldots, i_{r_k}^{(k)}, i$  have  repeated vertices. Suppose $i_{p'}^{(q')}$ ($1\leq p'\leq r_{q'}$)  is  the first listed vertex equal to, say $i_p^{(q)}$ ($1\leq p\leq r_q$). Note that $q'\geq q$ and $p'>p$ when  $q'= q$. By Claim 1,  $i_{p'}^{(q')}$ is the root of $e_q$ and $\beta_{q-1}=\beta_q$. However, this is a contradiction with the choice of $P_1$ as a shortest directed path from $\alpha$ to $\beta$ in $\mathfrak{D}$.

Therefore we conclude that $\beta=\gamma$, and from two cases above, we know that $i_1^{(1)}  e_1 i_1^{(2)}  e_{2}  \cdots i_1^{(k)}  e_{k} i$ is a path from 0 to $i$ in $\mathcal{T}$.
\end{proof}

\textbf{Claim 3.}  Any directed cycle in the strong component $C_{\alpha}$ must be of length $r$ and is determined by a fixed edge of $\mathcal{T}$.

\begin{proof}
Choose any given directed cycle $C:=\beta_1\beta_2\cdots \beta_{k}\beta_1$  in $C_{\alpha}$. If $\alpha$ is not in $C$, then assume, without loss of generality, that the distance from $\alpha$ to $\beta_1$ in $\mathfrak{D}$ is no more than that from $\alpha$ to other $\beta_i$, i.e. $\mathrm{dist}(\alpha, \beta_1)=\min\{\mathrm{dist}(\alpha, \beta_i): i=1,\ldots, k\}$.  Let $\alpha \alpha_1\cdots \alpha_l\beta_1$ be such a shortest directed path. Assume that
 \begin{equation*}
   \alpha \stackrel{(i_1,e_1)}{\longrightarrow}  \alpha_1\stackrel{(i_2,e_2)}{\longrightarrow} \alpha_2\rightarrow\cdots\rightarrow \alpha_l   \stackrel{(i_{l+1},e_{l+1})}{\longrightarrow} \beta_1\stackrel{(i_{l+2},e_{l+2})}{\longrightarrow} \beta_2\rightarrow\cdots\rightarrow \beta_k \stackrel{(i_{l+k+1},e_{l+k+1})}{\longrightarrow} \beta_1.
 \end{equation*}
 If $\alpha$ is in $C$, then a similar argument applies.

 By Claim 2, vertices $i_1, i_2, \ldots,   i_{l+2}$ are  distinct as $\alpha, \alpha_1, \ldots, \alpha_l, \beta_1$ are  distinct. Similarly, by Claim 2, vertices $i_{l+2}, \ldots, i_{l+k+1}$ are  distinct as $\beta_1, \ldots \beta_{k}$  are distinct. Further, by Claim 2 again, any one of $i_{l+3}, \ldots, i_{l+k+1}$ is different from any of $i_1, i_2, \ldots,   i_{l+1}$ because otherwise one of $\beta_2, \ldots \beta_{k}$  is closer  to $\alpha$.  Thus the vertices $i_1, i_2, \ldots,   i_{l+2}, \ldots, i_{l+k+1}$ are distinct.
 By Claim 1,    $e_{l+2 }=\cdots =e_{l+k+1}$,  $k=r$ and $e_{l+2 }=\{i_{l+2}, \ldots, i_{l+r+1}\}$. Therefore, $C$ is a cycle of length $r$ and is determined by the edge $e_{l+2}$.
\end{proof}

Now consider a natural mapping $\pi$ from the vertex set $V(C_{\alpha})$, consisting of all configurations strongly connected to $\alpha$, to the vertex set $[n]_0$ of $\mathcal{T}$ as follows: for any $\beta\in C_{\alpha}$, if $\beta\in S_i$, then $\pi(\beta)=i$.
The map $\pi$ is well-defined. Further, $\pi$ is injective by Claim 2. In fact, $\pi$ is a bijection, as demonstrated in the next claim.

 \textbf{Claim 4.} The mapping $\pi: V(C_{\alpha})\rightarrow [n]_0$ is a bijection. Further, $\pi^{-1}(v)|_{\hat{e}}=\alpha|_{\hat{e}}$ when $v$ is the root of $e$.
 \begin{proof}
 For any $v\in [n]_0$, by induction on $\mathrm{dist}(0, v)$,  we show that there exists $\beta\in C_{\alpha}$ such that  $\beta\in S_v$
  and $\beta|_{\hat{e}}=\alpha|_{\hat{e}}$ when $v$ is the root of $e$.  When $\mathrm{dist}(0, v)=0$, we have $v=0$ and then $\pi^{-1}(0)=\alpha$, $\pi^{-1}(0)|_{\hat{e}}=\alpha|_{\hat{e}}$ for any $e\in E_0$.

  Suppose the conclusion holds for any   vertex of distance $k-1$ from the root 0. For any vertex $v\in [n]_0$ with $\mathrm{dist}(0, v)=k$, there exists a unique path $P$ from the root $0$ to $v$ in $\mathcal{T}$, say $P=0e_1v_1e_2\cdots v_{k-1}e_k v$. Then $\mathrm{dist}(0, v_{k-1})=k-1$, and by the induction hypothesis, there exists   $\beta\in C_{\alpha}$ such that $\beta=\pi^{-1}(v_{k-1})$ and $\beta|_{\hat{e}}=\alpha|_{\hat{e}}$ whenever $v_{k-1}$ is the root of  $e$. Since $P$ is a path, $v_{k-1}$ is the root of  $e_{k}$ and so $((r-1)\one-\beta)|_{\hat{e}_{k}}$ is a parking function. By Claim 2, any unstable vertex with respect to $\beta$ must be in $e_1\cup \cdots \cup e_{k-1}$ and has labelling at most the labels in $e_{k}$, and so the incidences determined by $e_{k}$ are toppled contiguously. By Lemma~\ref{lem_unique-topple-order}, the ordering of toppling for incidences determined by $e_k$ is unique,  say
  \begin{equation*}
\beta\stackrel{(i_1,e_k)}{\longrightarrow} \beta_1\stackrel{(i_2,e_k)}{\longrightarrow} \beta_2\rightarrow\cdots\rightarrow \beta_{r-1}\stackrel{(i_r,e_k)}{\longrightarrow} \beta,
  \end{equation*}
where $e_k=\{i_1, i_2, \ldots, i_r\}$,  $i_1 =v_{k-1}$ is the root of $e_k$ and assume that $v=i_s$ ($2\leq s\leq r$). Then $\pi^{-1}(v)= \beta_{s-1}$ and $\beta_{s-1}$  is strongly connected to $\beta$, so $\beta_{s-1}\in C_{\alpha}$. For any $e\in E_v$, and if  $v$ is the root of $e$, then $e\notin\{e_1, \ldots, e_k\}$ and so $\beta_{s-1}|_{\hat{e}}=\alpha|_{\hat{e}}$.
  \end{proof}

By Claim 3, we know that any directed cycle in $C_{\alpha}$ has length $r$ and  is determined by a fixed edge of $\mathcal{T}$. Now we show the converse holds as well.

 \textbf{Claim 5.} Each edge of $\mathcal{T}$   determines one and  only one directed cycle of length $r$ in $C_{\alpha}$.
\begin{proof}
Choose any edge $e=\{i_1, i_2, \ldots, i_r\}$ of $\mathcal{T}$. Assume that $i_1$ is the root of  $e$ and let $\beta=\pi^{-1}(i_1)$.
By Claim 4, $\beta|_{\hat{e}}=\alpha|_{\hat{e}}$ and so $((r-1)\one-\beta)|_{\hat{e}}$ is a parking function. By Claim 2, we know that any possible  unstable vertex with respect to $\beta$ is no more than those in $e$ and so the incidences determined by $e$ can be toppled contiguously   from   $\beta$. By Lemma~\ref{lem_unique-topple-order}, the toppling sequence inside $e$ from $\beta$ is unique, say
\[
\beta\stackrel{(i_1,e)}{\longrightarrow} \beta_1 \stackrel{(i_2',e)}{\longrightarrow} \beta_2  \stackrel{(i_3',e)}{\longrightarrow}  \cdots \stackrel{(i_r',e)}{\longrightarrow} \beta.
\]
  Thus each edge $e$ of $\mathcal{T}$ determines a unique directed cycle  $\beta \rightarrow\beta_1\rightarrow\cdots\rightarrow\beta_{r-1}\rightarrow\beta$.
  \end{proof}

Now applying the above claims (namely, 3, 4, and 5), it follows that  $C_{\alpha}$ is indeed isomorphic to  a representative digraph of $\mathcal{T}$.  This completes the proof of Lemma \ref{lem_representative-digraph}.
\end{proof}

The next result concerns  the cycle structure associated with a toppled digraph of an $r$-tree with a particular assumed ordering of the vertices.
Compared with Example~\ref{exam1}, where cycles of length $kr$ occur for any $1\leq k\leq \Delta$, we conclude that the cycle structure of toppled digraphs rely heavily on the ordering of the non-root vertices of a rooted $\mathcal{T}$. Observe that the ordering of the vertices of the hyperstar in Example~\ref{exam1} is not a good ordering.
 Although we are not able to verify which orderings of the vertices guarantee the occurrence of certain types of cycles, we will find an ordering that guarantees the occurrence of only cycles of length $r$ as follows.

\begin{lemma}\label{cyclebyedge}
Any directed cycle of the toppled digraph associated to   an $r$-tree $\mathcal{T}$ with a good ordering on the vertex set has length of $r$ and is determined by an edge of $\mathcal{T}$.
\end{lemma}
\begin{proof}
Let $\mathcal{T}$ be an $r$-tree   on the vertex set  $[n]_0$ with a good ordering, and $\mathfrak{D}$ be the toppled digraph of $\mathcal{T}$ associated with this assumed ordering. For any given directed cycle $C$ of $\mathfrak{D}$, consider the toppling sequence $\mathcal{Z}$ of $C$ (viewed cyclically). By Lemma~\ref{lem_sametimestoppled}, there must exist a principal incidence in $\mathcal{Z}$.

\textbf{Claim 1.}  There is a proper toppling sequence of some principal incidence in $\mathcal{Z}$ such that its depth is equal to zero.
 \begin{proof}
 Choose arbitrarily a principal incidence, say $(v, e)$, in $\mathcal{Z}$. If the depth of some proper toppling sequence $\mathcal{Z}(v, e)$ of $(v, e)$ in  $\mathcal{Z}$ is equal to zero, there is nothing to prove.  Otherwise, there is a principal incidence $(u, f)\in  \mathcal{Z}(v, e)\setminus\{(v, e)\}$ such that the depth of $\mathcal{Z}(v, e)$ is equal to $\mathrm{dist}(v, u)$ in $\mathcal{T}$.   Now  we show that the depth of $\mathcal{Z}(u, f)$ is equal to zero. Suppose, by contradiction, that the depth of $\mathcal{Z}(u, f)$ is positive. Let $\mathcal{Z}_1=\mathcal{Z}(v, e)$.  From the proof of  Lemma~\ref{lem_properseq}, $\mathcal{Z}(u, f)=\mathcal{Z}_1(u, f)$ and so the depth of $\mathcal{Z}(v, e)$  should be no less than $\mathrm{dist}(v, u)$ plus the depth of $\mathcal{Z}_1(u, f)$, which implies that the depth of $\mathcal{Z}(v, e)$ is strictly larger than $\mathrm{dist}(v, u)$, a contradiction with the choice of $(u, f)$.
\end{proof}

Thus there is a principal incidence, say $(v, e)$, and a proper toppling sequence of $(v, e)$ in $\mathcal{Z}$, say $\mathcal{Z}(v, e)$, such that the depth of $\mathcal{Z}(v, e)$ is equal to zero.

\textbf{Claim 2.} $\mathcal{Z}(v, e)$ consists of $r$ distinct incidences determined by $e$.
\begin{proof}
By Lemma~\ref{lem_oncetopple}, each incidence determined by $e$ occurs once in $\mathcal{Z}(v, e)$ so it suffices to verify that no incidence determined by other edges occurs in $\mathcal{Z}(v, e)$. By contradiction, suppose that some incidence $(u, f)$ with $f\neq e$ occurs in $\mathcal{Z}(v, e)$ and,  without loss of generality, $(u, f)$ is the first such incidence in $\mathcal{Z}(v, e)$. Assume that
\[
\mathcal{Z}(v, e)=\{(v, e), \ldots, (v', e), (u, f), \ldots, (v'', e)\}
\]
and every incidence before $(u, f)$ in $\mathcal{Z}(v, e)$ is determined by $e$.

If $\mathfrak{L}(e)\leq \mathfrak{L}(f)$, i.e.  the level of $e$ is not greater than that of $f$, then   $(u, f)$ must be a principal incidence by   Corollary~\ref{prinfirst}.  Note that $(v, e)$ is the principal incidence   and  $(v'', e)$ is a non-principal incidence determined by $e$, and no $(v, e)$ occurs between $(u, f)$ and $(v'', e)$ in the proper toppling sequence $\mathcal{Z}(v, e)$. Then by Corollary~\ref{prinfirst} again, $\mathfrak{L}(f)\leq \mathfrak{L}(e)$ cannot hold, so $\mathfrak{L}(e)< \mathfrak{L}(f)$. Let $\alpha$ and $\beta$ be the configurations that
$(v, e)$ and $(u, f)$ follow in $\mathcal{Z}(v, e)$  respectively,  and let $z$ be the representative vector of the subsequence $\{(v, e), \ldots, (v', e)\} $ of $\mathcal{Z}(v, e)$. Note that $z\leq j_e$, namely $z$ is a $(0, 1)$-vector and the support of $z$ consists of some incidences determined by $e$.  Since $\mathfrak{L}(e)< \mathfrak{L}(f)$ and $\alpha\in S_v$, we have $\alpha(u)<r-1$.
As the depth of $\mathcal{Z}(v, e)$ is equal to zero, $f\notin E(\mathcal{T}_e)$ and thus
\begin{align*}
\beta(u)=\alpha(u)-(r-1)\sum_{h\in E_u}z(u, h)+\sum_{u\neq w\in h\in E_u}z(w, h)=\alpha(u),
\end{align*}
as $u\notin e$ and $z(w, h)=0$ for any $h\neq e$. This implies that $r-1\leq\beta(u)=\alpha(u)<r-1$, a contradiction.

Now consider the case when $\mathfrak{L}(e)>\mathfrak{L}(f)$. Note that $\mathcal{T}$ is the coalescence of $\mathcal{T}_e$ and  $\mathcal{T}_e'$.
Let $\alpha_1$ and $\beta_1$ be the configurations that
$(u, f)$ and $(v'', e)$ follow in $\mathcal{Z}(v, e)$,  respectively. Since $\mathfrak{L}(e)>\mathfrak{L}(f)$ and $\alpha_1\in S_u$, we have
\begin{itemize}
  \item $\alpha_1(a)<r-1$, for any $a\in \mathcal{T}_e\setminus\{v\}$;
   \item  $\beta_1(v'')\geq r-1$, where $v''\in \mathcal{T}_e\setminus\{v\}$.
\end{itemize}
Then by Lemma~\ref{cutvertex}, the incidence $(v, e)$ must occur before $(v'', e)$ in the subsequence $\{(u, f), \ldots, (v'', e)\}$ of $\mathcal{Z}(v, e)$, which is not possible because $\mathcal{Z}(v, e)$ is a proper toppling sequence of $(v, e)$ in $\mathcal{Z}$.
\end{proof}

Thus we may assume that  $\mathcal{Z}(v, e)=\{(v, e),   (v_2, e),   \ldots, (v_r, e)\}$, where $e=\{v,  v_2,    \ldots,  v_r\}$. Let $\gamma$ and $\gamma'$ be the configurations before and after $\mathcal{Z}(v, e)$ in $\mathcal{Z}$, respectively. Then $\gamma'=\gamma-Bz$, where $z$ is the representative vector of the toppling subsequence $\mathcal{Z}(v, e)$ of $\mathcal{Z}$. However, it is easy to see that $z=j_e$, and so $\gamma'=\gamma-0=\gamma$. Since $C$ is a directed cycle of $\mathfrak{D}$, $\mathcal{Z}(v, e)=\mathcal{Z}$ and $C$ is a cycle of length $r$ determined by the edge $e$.
\end{proof}

\begin{remark}
For each toppled digraph $\mathfrak{D}_i$ of $\mathcal{T}$, the cycle structure of $\mathfrak{D}_i$ and thus $\phi_{\mathfrak{D}_i}(\lambda)$ varies depending on the ordering of vertices in $V(\mathcal{T})\setminus\{i\}$. However, by Proposition~\ref{prop_GCD}, we know that the characteristic polynomial of $r$-tree $\mathcal{T}$ is the greatest common divisor of $\phi_{\mathfrak{D}_0}(\lambda), \phi_{\mathfrak{D}_1}(\lambda), \ldots, \phi_{\mathfrak{D}_n}(\lambda)$, which is invariant of the particular selection for each $\mathfrak{D}_i$. By Lemma~\ref{cyclebyedge}, the cycle structure of  a toppled digraph of $\mathcal{T}$ with a good ordering  is more straightforward and is better understood.  For our purposes, we always assume that the uniform hypertree under consideration is equipped with a good ordering, which helps simplify the analysis on the characteristic polynomial of a toppled digraph.
\end{remark}

We now state and prove one of our main observations.

\begin{theorem} \label{thm_main_charpoly}
Let $\mathcal{T}$ be an $r$-tree of size $m$ with $r\geq 2$.  Then
\[
\phi_{\mathcal{T}}(\lambda)=\prod_{H  \sqsubseteq \mathcal{T}}\varphi_{H}(\lambda)^{a_{H}}\cdot (\mathrm{extraneous\,\, factor})
\]
where the product is over all connected   subgraphs $H$ of $\mathcal{T}$. The exponent $a_{H}$  of the factor $\varphi_{H}(\lambda)$ can be written as
\[
a_H=b^{m-e(H)-|\partial(H)|}c^{e(H)}(b-c)^{|\partial(H)|},
 \]
where $e(H)$ is the size of $H$, $\partial(H)$ is the boundary of  $H$, and $b=(r-1)^{r-1}, c=r^{r-2}$.
\end{theorem}

\begin{proof}
Suppose the vertex set of $\mathcal{T}$ is denoted by $[n]_0$ with a good ordering according to the distance partition from the root 0. The variables are ordered consistently with the vertices as $x_i>x_j$ whenever $i>j$. Suppose the toppled digraph with respect to the root 0 is denoted by  $\mathcal{D}$. Now we prove that
\[
\prod_{H  \sqsubseteq \mathcal{T}}\varphi_{H}(\lambda)^{a_{H}}\,\,\,\mathrm{divides}\,\,\, \phi_{\mathcal{D}}(\lambda).
\]

To verify the equation above, choose an arbitrary critical configuration $\alpha$  on $\mathcal{T}$, and consider the strong component $C_{\alpha}$ of $\mathcal{D}$ containing $\alpha$. By Lemma~\ref{lem_representative-digraph}, $C_{\alpha}$   is isomorphic to a representative digraph of $\mathcal{T}$, and then   assume that
\begin{equation}\label{polystrongcomp}
  \phi_{C_{\alpha}}(\lambda)=\lambda^{n+1}+a_1\lambda^{n}+\cdots+a_{n+1}.
\end{equation}
By Proposition~\ref{prop_charpoly_digraph}, we have that  $a_i= \sum_{L\in \mathfrak{L}_i}(-1)^{p(L)}$, where $\mathfrak{L}_i$ is the set of  all linear directed subgraphs $L$ of $C_{\alpha}$ with exactly $i$ vertices and $p(L)$ denotes the number of components of $L$.

From the proof of Lemma~\ref{lem_representative-digraph}, it follows that   any directed cycle in $C_{\alpha}$ has length $r$ and is determined by some edge $e$ of $\mathcal{T}$,  denoted by $c_e$. More precisely, there is a one-to-one correspondence between all cycles of  $C_{\alpha}$ and all edges of $\mathcal{T}$ and this correspondence preserves intersection, i.e.
$c_{e}\cap c_{f}=e\cap f$, for any $e,  f\in E(\mathcal{T})$. Because   each cycle of $C_{\alpha}$ has length $r$, $a_i\neq 0$ in \eqref{polystrongcomp} only if $i$ is multiple of $r$, and so
 \[\phi_{C_{\alpha}}(\lambda)=\sum_{j \geq 0}a_{rj}\lambda^{n+1-rj}, \]
 where $a_{rj}$ is the number of  $j$ mutually disjoint $r$-cycles of  $C_{\alpha}$, which in turn is equal to the  number of $j$-matchings in
 $\mathcal{T}$, or  $m_{\mathcal{T}}(j)$.

  Thus
\[
\phi_{C_{\alpha}}(\lambda)=\sum_{j \geq 0}a_{rj}\lambda^{n+1-rj}=\sum_{j \geq 0}m_{\mathcal{T}}(j)\lambda^{n+1-rj}=\varphi_{\mathcal{T}}(\lambda).
\]
Now we count all such critical configurations $\alpha$ on $\mathcal{T}$. Recall that for any critical configuration $\alpha$ on $\mathcal{T}$, we have that $((r-1)\one-\alpha)|_{\hat{e}}$ is a parking function of length $r-1$ for every $e\in E(\mathcal{T})$, and all $\alpha|_{\hat{e}}$ can be chosen randomly and independently in the set $\{\gamma \,\,:\,\, (r-1)\one-\gamma\in \mathcal{P}_{r-1}\}$, of which the cardinality is  $r^{r-2}$.
 Hence there are $r^{(r-2)m}$ such critical configurations on $\mathcal{T}$ which belong to $S_{0}$, and each belongs to a different strong component of $\mathcal{D}$ by Claim 2 in Lemma~\ref{lem_representative-digraph}.
 Since  $\phi_{\mathcal{D}}(\lambda)$ is a product of the characteristic polynomials of the strong components of $\mathcal{D}$,  we have $(\phi_{C_{\alpha}}(\lambda))^{r^{(r-2)m}}|\phi_{\mathcal{D}}(\lambda)$, and then conclude that
\[
\varphi_{\mathcal{T}}(\lambda)^{r^{(r-2)m}}\,\,\,\mathrm{divides}\,\,\, \phi_{\mathcal{D}}(\lambda).
\]

Choose arbitrarily a fixed connected   subgraph $H$ of  $\mathcal{T}$. Assume that $u$ is the least vertex in $H$ with respect to the ordering on $V(\mathcal{T})$, and select $u$ as the root of $H$.  Consider the boundary $\partial(H)$ of $H$. [Note that $\partial(H)=\emptyset$ when $H=\mathcal{T}$ and $\partial(H)=E_u$ when $H$ consists of a single vertex $u$.]

Now consider a connected proper subgraph $H$ of  $\mathcal{T}$. Then   $\partial(H) \neq\emptyset$.  Assume that  $\partial(H)=\{e_1, \ldots, e_h\}$, where $h$ is a positive integer. Each boundary edge $e_i$   contains a unique common vertex with    $H$,   called a \textit{boundary vertex} of $H$ and denoted by $v_i$. Set $\hat{e}_i=e_i\backslash \{v_i\}$.  Suppose we select $u$ as the root of $\mathcal{T}$ and consider the good ordering determined by the distance partition from the new root $u$. Then $v_i$ becomes the root of $e_i$ for $i=1, \ldots, h$. Observe then that the root of each edge in $H$ remains unchanged according to this new good ordering on $\mathcal{T}$.

Consider any configuration $\alpha$ on $\mathcal{T}$ satisfying the following conditions:
\begin{align}\label{conditionalphaH}
((r-1)\one-\alpha)|_{\hat{e}}\in \mathcal{P}_{r-1},\quad &  \forall e\in E(H);\nonumber \\
  0<((r-1)]\one-\alpha)|_{\hat{e}_i}\notin \mathcal{P}_{r-1},\quad & \forall  e_i\in\partial(H); \\
\alpha(v)\in\{0, 1, \ldots, r-2\},\quad & \forall v\notin H\cup \partial(H). \nonumber
\end{align}

From the form of $\alpha$, we know that  $\alpha(w)<r-1$ for any $w\neq u$.  For any $\beta\in\mathcal{D}$,   consider    directed paths in $\mathcal{D}$ from  $\alpha$ to $\beta$ when  $\beta$  is reachable from $\alpha$.  There are two cases to consider.

\textbf{ Case 1.} There exists a  directed path in $\mathcal{D}$ from  $\alpha$ to $\beta$ whose toppling sequence only uses edges from $H$, and denote by $U$ the set of all such $\beta$. Note that $\alpha\in U$. From \eqref{conditionalphaH} and the choice of $H$, we know that
\begin{itemize}
  \item $((r-1)\one-\alpha)|_{\hat{e}}$ is a parking function for every $e\in E(H)$.
  \item  $\alpha(w)<r-1$ for every $w\notin H$.
  \item  $\alpha\in S_u$ and $u$ is the root of $H$.
\end{itemize}

 Adopting the same method as in Lemma~\ref{lem_representative-digraph}, we can show that
 \begin{itemize}
 \item The mapping $\pi: U\rightarrow V(H)$ defined by $\pi(\beta)=i$ for $\beta\in S_i$ is a bijection. Further, $\pi^{-1}(v)|_{\hat{e}}=\alpha|_{\hat{e}}$ when $v\in H$ is the root of $e\in E(H)$.
     \item  Any directed cycle in $\mathcal{D}[U]$ has length   $r$ and is determined by a fixed edge of $H$.  Each edge of $H$   determines one and  only one directed cycle of length $r$ in $\mathcal{D}[U]$.
    \end{itemize}
Then the   subgraph of $\mathcal{D}$ induced by $U$ is isomorphic to a representative digraph of $H$.

\textbf{ Case 2.} Any directed path from  $\alpha$ to $\beta$ in $\mathcal{D}$ topples inside an edge not in  $E(H)$.
  Then one of the  edges in $\partial(H)$ must be used in any  directed path in $\mathcal{D}$ from  $\alpha$ to $\beta$.  This can be shown by contradiction. Consider  a  directed path in $\mathcal{D}$ from  $\alpha$ to $\beta$, say
\begin{equation}\label{alphabetapath}
  \alpha\stackrel{(i_1,f_1)}{\longrightarrow} \alpha_1\rightarrow\cdots  \stackrel{(i_{s-1},f_{s-1})}{\longrightarrow}  \alpha_{s-1}   \stackrel{(i_s, f_s)}{\longrightarrow} \alpha_s\rightarrow\cdots \rightarrow\beta,
\end{equation}
and assume that $f_s$ is the first one not in $E(H)$ for this toppling sequence. Suppose that $f_s\notin \partial(H)$.  Note that $i_1=u$ and then $f_1\in E_{u}\setminus\partial(H)$, which implies that $s\geq2$. By the choice of $f_s$, we know that $f_1, \ldots, f_{s-1}\in E(H)$.
Since $\alpha\in S_u$ and it is reducible, all possible unstable vertices with respect to $\alpha_{s-1}$ must be in   $f_1\cup \cdots\cup f_{s-1}$. By assumption, $\alpha_{s-1}\in S_{i_s}$ and so $i_s\in f_1\cup \cdots\cup f_{s-1}$. However, $i_s\in f_s$ and $f_s\notin E(H)\cup\partial(H)$ imply that $i_s\notin H$, a contradiction. Therefore   $f_s\in\partial(H)$, say  $f_s=e_1$ without loss of generality.  Since $f_1, \ldots, f_{s-1}\in E(H)$, we know that $\alpha_{s-1}\in U$ and $\alpha_{s-1}|_{\hat{e}_{1}}=\alpha|_{\hat{e}_{1}}$.

Suppose by contradiction that $\alpha$ is reachable from $\beta$. Together with \eqref{alphabetapath}, $\alpha$ and $\beta$ are strongly connected and so  there is a cycle $C$ containing $\alpha_{s-1}$, given by
\begin{equation}\label{r-cycle}
 \alpha_{s-1}   \stackrel{(v_1, e_1)}{\longrightarrow} \alpha_s\rightarrow\cdots \rightarrow \alpha_{s-1}.
\end{equation}
By Lemma~\ref{cyclebyedge},  $C$ is a cycle of length $r$ and is determined by the edge $e_1$.

If $e_{1}$ is a lower edge for $u$, then $\alpha_{s-1}|_{\hat{e}_1}=\alpha|_{\hat{e}_1}$ is a stable configuration on the complete 2-graph $K_r$ for chip-firing game, where $K_r$  has $e_1$ as its vertex set and the boundary vertex $v_1$  as its root.
 From \eqref{r-cycle}, $\alpha_{s-1}|_{\hat{e}_1}$ is also a recurrent configuration, and hence  is a critical configuration on $K_r$. By Lemma~\ref{prop_parkingfunc}, $((r-1)\one-\alpha_{s-1})|_{\hat{e}_{1}}$ must be  a parking function, which contradicts  the requirement in \eqref{conditionalphaH}.

If $e_{1}$ is an upper edge for $u$,  then $v_1=u$, which is not the lowest labelled vertex in $e_1$,  and $e_1$ is the only upper edge for $u$ in $E_u$.  Assume that $e_{1}=\{u_{1}, u_2,\ldots, u_{r}\}$ and $u_r$  is the least vertex in $e_1$. Then $v_1\in\{u_{1},  \ldots, u_{r-1}\}$.  
Now consider the labelled complete 2-graph $K_r$ on the vertex set $\{u_{1}, u_2,\ldots, u_{r}\}$    rooted at $u_r$.
From \eqref{r-cycle},  $\alpha_{s-1}|_{e_1\setminus\{u_r\}}$ is a recurrent configuration on $K_r$,  by   Theorem~\ref{confluenceroot},  after toppling some incidences determined by $e_1$ from  $\alpha_{s-1}$,      a unique critical configuration on $K_r$, say $\gamma$,  will be obtained. Further by Lemma~\ref{prop_parkingfunc},  $((r-1)\one-\gamma)|_{e_{1}\setminus\{u_r\}}$ is  a parking function. Note that $\gamma$ can result in $\alpha_{s-1}$ after completing the topplings along the remaining incidences determined by $e_1$.
Since  $((r-1)\one-\gamma)|_{e_{1}\setminus\{u_r\}}$ is  a parking function,  we may suppose that
\[
\gamma(u_r)\geq r-1, \quad \gamma(u_{1})\geq r-2, \cdots, \gamma(u_{r-2})\geq 1, \quad \gamma(u_{r-1})\geq 0.
\]
After toppling the incidence $(u_r, e_{1})$, $\gamma$ becomes a configuration, say $\gamma'$, which satisfies
 \[
\gamma'(u_r)\geq 0, \quad \gamma'(u_1)\geq r-1, \cdots, \gamma'(u_{r-2})\geq 2, \quad \gamma'(u_{r-1})\geq 1.
\]
Continuing with the remaining topplings from $\gamma$ to $\alpha_{s-1}$, we have that $\alpha_{s-1}(v_{1})\geq r-1$ and with respect to  $\alpha_{s-1}$, one of $\hat{e}_{1}$ is  $\geq 0$, one of $\hat{e}_{1}$ is  $\geq 1$, \ldots, one of $\hat{e}_{1}$ is  $\geq r-2$.
As $\alpha_{s-1}(a)<r-1$ for all $v_1\neq a\in e_1$, this together with the   above conclusion implies that $((r-1)\one- \alpha_{s-1})|_{\hat{e}_{1}}$ is  a parking function,  which contradicts  the requirement in \eqref{conditionalphaH}.  Therefore  $\alpha$ is not  reachable from   $\beta$.

 Based on the analysis of the two cases above,  we conclude that the strong component $C_{\alpha}$   is isomorphic to a representative digraph of  $H$ and hence $\phi_{C_{\alpha}}(\lambda)=\varphi_{H}(\lambda)$.  Further from the construction of $\alpha$, we have
 \begin{itemize}
  \item $\alpha|_{\hat{e}}$ has $r^{r-2}$ choices, for each $e\in E(H)$;
  \item $\alpha|_{\hat{e}}$ has $(r-1)^{r-1}-r^{r-2}$ choices, for each $e\in \partial(H)$;
  \item $\alpha|_{\hat{e}}$ has $(r-1)^{r-1}$ choices, for each $e\notin E(H)\cup\partial(H)$.
\end{itemize}

Since $\alpha|_{\hat{e}}$ can be chosen  independently,   there are $a_H$ such configurations $\alpha$ on $\mathcal{T}$  all lying in different strong components of $\mathcal{D}$ as they all belong to $S_{u}$, and
$$a_H=((r-1)^{r-1})^{m-e(H)-|\partial(H)|}(r^{r-2})^{e(H)}(r-1)^{r-1}-r^{r-2})^{|\partial(H)|}.$$  Therefore, we have
$\varphi_{H}(\lambda)^{a_H}\,\,\,\mathrm{divides}\,\,\, \phi_{\mathcal{D}}(\lambda).$

 For any two connected  subgraphs $H_1, H_2$ of $\mathcal{T}$, let $\alpha_i$   be any  configuration  on $\mathcal{T}$ satisfying the requirement \eqref{conditionalphaH} with $H$ replaced by $H_i$, for $i=1, 2$.   Now we show that $H_1\neq H_2$ implies that $C_{\alpha_1}\neq C_{\alpha_2}$. Since $H_1\neq H_2$, either  $H_1\setminus H_2\neq\emptyset$ or $H_2\setminus H_1\neq\emptyset$ and, without loss of generality, assume that there is a vertex, say $u_1$ in $H_1\setminus H_2$. Consider the bijection as in Case 1. There is a configuration $\beta\in C_{\alpha_1}$ such that $\beta(u_1)\geq r-1$ and $\gamma(u_1)< r-1$ for any $\gamma\in C_{\alpha_2}$, which implies that  $\beta\notin C_{\alpha_2}$. Therefore we arrive at
 \[
\prod_{H  \sqsubseteq \mathcal{T}}\varphi_{H}(\lambda)^{a_{H}}  \,\,\,\mathrm{divides}\,\,\,  \phi_{\mathcal{D}}(\lambda).
\]
Observe that the above conclusion does not depend on the choice of the root of $\mathcal{T}$. Therefore  $\prod_{H \sqsubseteq \mathcal{T}}\varphi_{H}(\lambda)^{a_{H}}$ is a common divisor of the characteristic polynomials of the toppled digraphs and hence divides their greatest common divisor, which is the characteristic polynomial of $\mathcal{T}$,  by Proposition~\ref{prop_GCD}.
\end{proof}

In fact, it can be shown that the extraneous factor is equal to one in the above theorem, and hence we conclude the section with our main result.

\begin{theorem}\label{thm_main}
For any $r$-tree $\mathcal{T}$ of size $m$ with $r\geq2$,
\begin{equation}\label{mainidentity}
\phi_{\mathcal{T}}(\lambda)= \prod_{H \sqsubseteq \mathcal{T}}\varphi_{H}(\lambda)^{a_{H}},
\end{equation}
where the product runs over all connected  subgraphs $H$ of $\mathcal{T}$, and
\[
a_H=b^{m-e(H)-|\partial(H)|}c^{e(H)}(b-c)^{|\partial(H)|},
 \]
with $e(H)$ being the size of $H$, $\partial(H)$ being the boundary of  $H$, and $b=(r-1)^{r-1}, c=r^{r-2}$.
\end{theorem}

\begin{proof}
For convenience, denote by $f(\lambda)$ the polynomial in the right side of \eqref{mainidentity}, i.e. $f(\lambda):=\prod_{H \subseteq \mathcal{T}}\varphi_{H}(\lambda)^{a_{H}}$.
By Theorem~\ref{thm_main_charpoly}, it suffices to show that $f(\lambda)$ has the same degree as $\phi_{\mathcal{T}}(\lambda)$ since the former is a factor of the latter, and both are monic. By Proposition~\ref{prop_ratiodet}, we know the degree of the characteristic polynomial $\phi_{\mathcal{T}}(\lambda)$ is equal to the number of reduced monomials of total degree $d$. Denote by $X$ the set consisting of all reduced monomials of total degree $d$. Next we verify that the degree of  $f(\lambda)$ is $|X|$.

For any $x^{\alpha}\in X$, $\alpha$ is a  configuration on $\mathcal{T}$, and assume that $x^{\alpha}\in S_i$.
 For each edge $e$ of $\mathcal{T}$, denote by $\hat{e}$ the one obtained from $e$ by deleting its root with respect to the distance partition determined by $i$. Consider all maximal components of $\mathcal{T}$ induced by  all such edges $e$ satisfying $((r-1)\one-\alpha)|_{\hat{e}}\in \mathcal{P}_{r-1}$.
  Let $H$ denote  the maximal component containing $i$. If there is no edge $e$ incident to $i$ that satisfies   $((r-1)\one-\alpha)|_{\hat{e}}\in \mathcal{P}_{r-1}$,  then $H$ consists of  single vertex $i$.
Denote by $Y$ the set of all connected subgraphs of $\mathcal{T}$. There is a natural mapping $\sigma: X\rightarrow Y$ defined by $\sigma(x^{\alpha})=H$, which is well-defined from the correspondence above.

Observe that  $0<((r-1)\one-\alpha)|_{\hat{e}}\notin \mathcal{P}_{r-1}$ for any $e\in \partial(H)$ because of the maximality of  $H$.
 Given $H\in Y$ and $v\in H$, there is $a_H$ such $x^{\alpha}$ in $X\cap S_v$ with $\sigma(x^{\alpha})=H$.  Thus  for any $H\in Y$,
\[
|\sigma^{-1}(H)|=a_{H}|V(H)|,
\]
and so
\[
|X|=\sum_{H\in Y}|\sigma^{-1}(H)|=\sum_{H\in Y}|V(H)|a_{H},
\]
where the last term is just the degree of $f(\lambda)$.
\end{proof}

When $r=2$ in Theorem~\ref{thm_main}, $b-c=(r-1)^{r-1}-r^{r-2}=0$ and
\[
(b-c)^{|\partial(H)|}=\left\{
                        \begin{array}{ll}
                          0, & \hbox{if $H\sqsubseteq \mathcal{T}$ and $H\neq \mathcal{T}$;} \\
                          1, & \hbox{if $H=\mathcal{T}$.}
                        \end{array}
                      \right.
\]
Therefore for  $r=2$, the  correspondence in Theorem~\ref{thm_main} reduces to  the classical result Theorem~\ref{forestcharmatchpoly} for ordinary trees.

A basic consequence of  Theorem~\ref{thm_main} is a combinatorial method to compute the characteristic polynomial of hypertrees, which is difficult due to the high complexity of computing the resultant from various techniques in algebraic geometry, and, in addition, this result constructs a strong bridge between an algebraic and a combinatorial  object. Thus it deepens the understanding of possible interactions between these two important polynomials on hypertrees.

\begin{figure}[!hbpt]
\begin{center}
\includegraphics[scale=0.8]{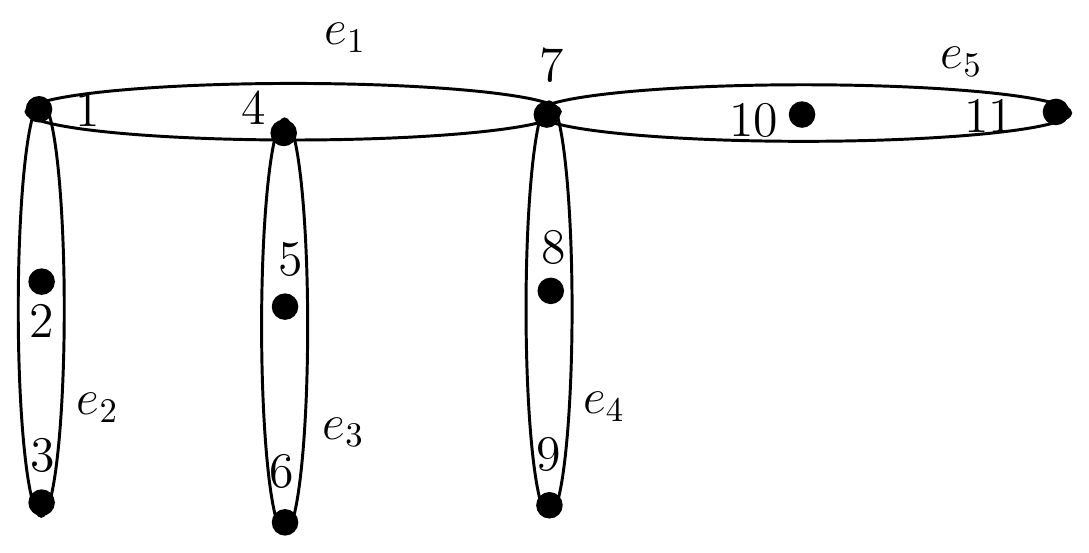}
\caption{The hypergraph $\mathcal{H}$:  3-tree of order 11.}
\end{center}\label{fig_6cycle}
\end{figure}

\begin{example}
We illustrate the conclusion of Theorem \ref{thm_main} with the following example. Take the 3-uniform hypergraph  $\mathcal{H}=([11], E)$ from \cite{ClarkCoop21} see Figure 2, where $E=\{e_1, e_2 ,e_3, e_4, e_5\}$.
 From the computation in  \cite{ClarkCoop21}, it follows that
\begin{align}\label{chpoly}
 \phi_{\mathcal{H}}(\lambda)&=\lambda^{2192}(\lambda^{11}-5\lambda^8+5\lambda^5-2\lambda^2)^{243}(\lambda^9-4\lambda^6+3\lambda^3-1)^{162}
(\lambda^9-4\lambda^6+2\lambda^3)^{162}\nonumber\\
&\qquad (\lambda^7-3\lambda^4+\lambda)^{135}(\lambda^7-3\lambda^4)^{27}(\lambda^5-2\lambda^2)^{180}(\lambda^3-1)^{483}.
\end{align}

\begin{table}[!hbpt]
\begin{center}
 \begin{tabular}{|c|c|c|}
   \hline
  $H\sqsubseteq \mathcal{H}$ & $a_H$ &   $\varphi_{H}(\lambda)$ \\
    \hline
  $\{2\}, \{3\}, \{5\}, \{6\}, \{8\}, \{9\}, \{10\}, \{11\}$ & 256     & $\lambda$\\
   \hline
   $\{1\}, \{4\}$ &64 & $\lambda$\\
   \hline
    $\{7\}$ & 16 &$\lambda$\\
   \hline
    $\{e_2\}, \{e_3\}$ & 192&$\lambda^3-1$\\
   \hline
   $\{e_4\}, \{e_5\}$ & 48 &$\lambda^3-1$\\
   \hline
    $\{e_1\}$ & 3 &$\lambda^3-1$\\
   \hline
    $\{e_1, e_2\}, \{e_1, e_3\}, \{e_1, e_4\}, \{e_1, e_5\}$& 9 &$\lambda^5-2\lambda^2$\\
   \hline
   $\{e_4, e_5\}$ &144 &$\lambda^5-2\lambda^2$\\
   \hline
   $\{e_1, e_2, e_3\}, \{e_1, e_2, e_4\}, \{e_1, e_2, e_5\}, \{e_1, e_3, e_4\}, \{e_1, e_3, e_5\}$ &27 & $\lambda^7-3\lambda^4+\lambda$\\
   \hline
   $\{e_1,  e_4, e_5\}$ &27 & $\lambda^7-3\lambda^4$\\
   \hline
   $\{e_1,  e_2, e_3, e_4\}, \{e_1,  e_2, e_3, e_5\}$ &81 & $\lambda^9-4\lambda^6+3\lambda^3-1$\\
   \hline
   $\{e_1,  e_2, e_4, e_5\}, \{e_1,  e_3, e_4, e_5\}$ &81 & $\lambda^9-4\lambda^6+2\lambda^3$\\
   \hline
    $\mathcal{H}$ &243 &$\lambda^{11}-5\lambda^8+5\lambda^5-2\lambda^2$\\
   \hline
 \end{tabular}
\caption{All connected subgraphs $H$ of $\mathcal{H}$ together with  $a_H$ and  $\varphi_{H}(\lambda)$.}
\end{center}\label{fig_6cycle}
\end{table}
All connected subgraphs $H$ of $\mathcal{H}$ listed in Table 1,  are presented via the vertex induced subsets or edge induced subsets, together with their exponents $a_H$ and matching polynomials. From \eqref{chpoly} and the data in Table 1, it is straightforward to verify that the following equation holds:
\[
\phi_{\mathcal{H}}(\lambda)= \prod_{H  \sqsubseteq \mathcal{H}}\varphi_{H}(\lambda)^{a_{H}},
\]
where the product runs over all connected subgraphs $H$ of $\mathcal{H}$.
\end{example}

\section{Applications}

Let $G=(V,E)$ be an ordinary  graph (i.e. a 2-graph). For every $r\geq3$, the {\em $r$th power }of $G$, denoted by $G^r$, is an $r$-uniform hypergraph with vertex set $V(G^r)=V\cup(\cup_{e\in E}\{i_{e,1},\ldots, i_{e,r-2}\})$ and edge set
  $E(G^r)=\{e\cup\{i_{e,1},\ldots, i_{e,r-2},\}|~e\in E\}$.   Let $P_m$   denote the path  with $m$ edges. The $r$th power of $P_m$, denoted by $P_{m}^r$,  is called a {\em loose path}.

\begin{prop}[\cite{SuKLS}]\label{prop-tree-powertree-matchpoly}
Let $T$ be a 2-tree on $n$ vertices,   $r \, (r\geq3)$ a positive integer. Then the matching polynomials of $T$ and   $T^r$ satisfy the following relation:
\[
\varphi_{T^r}(\lambda)=\lambda^{\frac{(n-2)(r-2)}{2}}\varphi_{T}(\lambda^{\frac{r}{2}}).
\]
\end{prop}

 Chen and Bu~\cite{ChenBu} derived   an explicit formula for the characteristic polynomial   of uniform loose paths.  Combining Theorem~\ref{thm_main}, Proposition~\ref{prop-tree-powertree-matchpoly} and Theorem~\ref{forestcharmatchpoly},  the following  result from \cite{ChenBu} can be obtained as a consequence.

\begin{coro}[\cite{ChenBu}]
The characteristic polynomial of the $r$-uniform loose path $P_{m}^r$ of length $m$ is
\[
 \phi_{P_{m}^r}(\lambda)=\prod_{j=0}^m \phi_{P_{j}}(\lambda^{r/2})^{a(j, m)},
\]
where
\[
a(j, m)=\left\{
          \begin{array}{ll}
            K_2^m, & \hbox{j=m;} \\
           ((m-j+1)K_1+2K_2)K_1K_2^j(K_1+K_2)^{m-j-2}, & \hbox{$1\leq j\leq m-1$;} \\
             \frac{2}{r}[m(r-1)+](K_1+K_2)^{m}-\sum_{s=1}^m(s+1)a(s,m), & \hbox{j=0.}
          \end{array}
        \right.
\]
Here $K_1=(r-1)^{r-1}-r^{r-2}$, $K_2=r^{r-2}$ and $\phi_{P_{j}}(\lambda)$ is the characteristic polynomial of the path $P_j$.
\end{coro}

Another important fact from the literature along these lines is a formula for the characteristic polynomial of starlike hypergraphs by Bao et al. in \cite{BaoFanWang}, which can be expressed succinctly by applying Theorem~\ref{thm_main}.

Zhang et al. \cite{Zhang_17}   introduced a polynomial  $\varphi_{\mathcal{H}}:=\sum\limits_{k\geq 0}(-1)^{k}m_{\mathcal{H}}(k)x^{(\nu(\mathcal{H})-k)r}$ of an $r$-uniform hypertree $\mathcal{H}$ and showed that if a scalar is a nonzero eigenvalue of  $\mathcal{H}$ with an
eigenvector  having all elements nonzero, then   it is a root of  $\varphi_{\mathcal{H}}$. Clark and Cooper~\cite{ClarkCooper18} called the  polynomial $\varphi_{\mathcal{H}}$,  the matching polynomial of $\mathcal{H}$ and observed  that for any hypertree,  its matching polynomial divides  the characteristic polynomial and, consequently, posed the following conjecture.

\begin{conj}
  If $H$ is a   subgraph (induced by some vertex subset)  of  an $r$-tree $\mathcal{H}$ for $r\geq 3$, then $\varphi_{H}|\phi_{\mathcal{H}}$ and, in particular, $\phi_{H}|\phi_{\mathcal{H}}$.
\end{conj}

Next we observe that the second claim in the above conjecture holds assuming certain conditions. For this we need the  following result.

\begin{theorem}[\cite{CoopDut12}] \label{thm_UNION}
Let $H$ be an  $r$-graph that is the disjoint union of hypergraphs $H_1$ and $H_2$. Then
\[
 \phi_{H}(\lambda)= \phi_{H_1}(\lambda)^{(r-1)^{|H_2|}}\phi_{H_2}(\lambda)^{(r-1)^{|H_1|}}.
\]
\end{theorem}

A counterexample for the second claim in the above conjecture is  presented in the example below.

\begin{example}
Consider the 3-graph  $\mathcal{H}=([11], E)$ in Figure 2. Consider the subgraph $\mathcal{H}':=\mathcal{H}-7$, obtained from   $\mathcal{H}$ by deleting the vertex 7.  Note that $\mathcal{H}'$ consists of six components, four of which are a single vertex and two of which are a single edge.

Applying Theorem~\ref{thm_UNION} and Theorem~\ref{thm_main}, by a straight computation, we have
\begin{align*}
  \phi_{\mathcal{H}'}(\lambda) &=  \prod  \phi_{H}(\lambda)^{(r-1)^{|\overline{H}|}}=\lambda^{2816}(\lambda^3-1)^{768},
\end{align*}
where the product is over all components $H$ of $\mathcal{H}'$, and $\overline{H}$ is the set of all vertices in $\mathcal{H}'$ but not in $H$.
Obviously,  $\phi_{\mathcal{H}'}(\lambda)$ doesn't divide $ \phi_{\mathcal{H}}(\lambda)$ as given in \eqref{chpoly}.
\end{example}

We now restate the above conjecture and  provide a proof relying on Theorem~\ref{thm_main}.

\begin{coro}\label{coop-thm}
Suppose $H$ is a  subgraph (induced by some vertex subset) of  an $r$-tree $\mathcal{T}$ for some $r\geq 3$ Then $\varphi_{H}|\phi_{\mathcal{T}}$. Furthermore we   have that  $\phi_{H}|\phi_{\mathcal{T}}$ if either $r\geq4$ or $H$ is connected when $r=3$.
\end{coro}
\begin{proof}
For the first statement, it suffices to show the validity of $\varphi_{H}|\phi_{\mathcal{T}}$ for the matching polynomial $\varphi_{H}$ as defined in \eqref{e-matchingpoly} and for a connected subgraph  $H$ of $\mathcal{T}$ by Theorem~\ref{thm_matchingpoly}.
It  is easy to prove that for any $r\geq3$,
\[
b-c=(r-1)^{r-1}-r^{r-2}>0.
\]
Thus for any connected subgraph $H$ of $r$-tree $\mathcal{T}$  with $r\geq3$,
\[
a_H=(r-1)^{(r-1)(m-e(H)-|\partial(H)|)}r^{(r-2)e(H)}((r-1)^{r-1}-r^{r-2})^{|\partial(H)|}>0,
\]
and so $\varphi_{H}(\lambda)|\phi_{\mathcal{T}}(\lambda)$, by Theorem~\ref{thm_main}.

Now we prove the second statement. Assume first that $r\geq4$.  It suffices to show that   $\phi_{\mathcal{T}-v}(\lambda)|\phi_{\mathcal{T}}(\lambda)$ for any $v\in \mathcal{T}$.
For any connected subgraph $H$ of   $\mathcal{T}-v$, use $a'_H$ and  $a_H$ to denote the exponent of  $\varphi_{H}(\lambda)$  as a factor in   $\phi_{\mathcal{T}-v}(\lambda)$ and $\phi_{\mathcal{T}}(\lambda)$, respectively. Denote by $H_1$ the component to which $H$ belongs in $\mathcal{T}-v$, and by $H_2$ the union of other components  of $\mathcal{T}-v$. By Theorem~\ref{thm_UNION}, we have
\[
 \phi_{\mathcal{T}-v}(\lambda)= \phi_{H_1}(\lambda)^{(r-1)^{|H_2|}}\phi_{H_2}(\lambda)^{(r-1)^{|H_1|}}.
\]
Together  with Theorem~\ref{thm_main},  we know that
\begin{align*}
  a'_H &= (r-1)^{(r-1)(e(H_1)-e(H)-|\partial'(H)|)+|H_2|}c^{e(H)}(b-c)^{|\partial'(H)|},  \\
a_H &= (r-1)^{(r-1)(m-e(H)-|\partial(H)|)}c^{e(H)}(b-c)^{|\partial(H)|},
 \end{align*}
where $\partial'(H)$ is the boundary of $H$ in $H_1$, and $m$ is the size of $\mathcal{T}$.  Since $\mathcal{T}$ has no cycles, $v$ is incident to at most one edge in the boundary $\partial(H)$, and then $|\partial(H)|-|\partial'(H)|=1$ or 0 depending on if $v$ is in $\partial(H)$ or not.

Note that $|H_1|+|H_2|=|\mathcal{T}|-1$. By Lemma~\ref{hypertree-order-size},
\[
(r-1)e(H_1)+|H_2|=|H_1|-1+|H_2|=|\mathcal{T}|-2=(r-1)m-1,
\]
 and so
\[
\frac{a'_H}{a_H}=\left\{
                   \begin{array}{ll}
                     \frac{1}{r-1}, & \hbox{if $|\partial(H)|-|\partial'(H)|=0$;} \\
                      \frac{(r-1)^{r-2}}{(r-1)^{r-1}-r^{r-2}}, & \hbox{if $|\partial(H)|-|\partial'(H)|=1$.}
                   \end{array}
                 \right.
\]
Obviously $\frac{1}{r-1}<1$ for $r\geq4$. Observe that  when $r=4$, $\frac{(r-1)^{r-2}}{(r-1)^{r-1}-r^{r-2}}=\frac{9}{11}<1$ and when $r\geq5$,
\[
 \frac{(r-1)^{r-2}}{(r-1)^{r-1}-r^{r-2}}= \frac{1}{r-1-(1+\frac{1}{r-1})^{r-2}}< \frac{1}{r-1-e}<1,
\]
where $e$ is the base of the natural logarithm.
Thus for any subgraph $H$ of $r$-tree $\mathcal{T}$ with $r\geq4$, we have $a'_H\leq a_H$ and so $\phi_{\mathcal{T}-v}(\lambda)|\phi_{\mathcal{T}}(\lambda)$. Since $v$ was arbitrarily chosen, the characteristic polynomial of any subgraph of  $\mathcal{T}$ divides that of  $\mathcal{T}$ when $r\geq4$.

Assume now that $H$ is a connected subgraph of $\mathcal{T}$.  For any connected subgraph $H_1$ of $H$, use $a_{H_1}$ and $b_{H_1}$ to denote the exponent of $\varphi_{H_1}(\lambda)$ as a factor in $\phi_{\mathcal{T}}(\lambda)$ and $\phi_{H}(\lambda)$, respectively.
By Theorem~\ref{thm_main},
\begin{align*}
  a_{H_1} &= b^{m-e(H_1)-|\partial(H_1)|}c^{e(H_1)}(b-c)^{|\partial(H_1)|},  \\
b_{H_1}  &= b^{e(H)-e(H_1)-|\partial'(H_1)|}c^{e(H)}(b-c)^{|\partial'(H_1)|},
 \end{align*}
where $\partial'(H_1)$ is the boundary of $H_1$ in $H$, $m$ is the size of $\mathcal{T}$. Note that $\partial'(H_1)\subseteq \partial(H_1)$ and $E(H)\backslash(E(H_1)\cup\partial'(H_1))\subseteq E(\mathcal{T})\backslash(E(H_1)\cup\partial(H_1))$, and so $ a_{H_1}\geq b_{H_1} $, which implies that $\varphi_{H_1}(\lambda)^{b_{H_1} }|\phi_{\mathcal{T}}(\lambda)$.  Therefore  $\phi_{H}(\lambda)|\phi_{\mathcal{T}}(\lambda)$ since $H_1$ is an arbitrarily chosen subgraph of $H$.
\end{proof}

The nullity of an ordinary 2-graph has been investigated extensively as it is an important algebraic parameter for graphs, and has intriguing applications in chemistry and physics (see \cite{WangGuo} and references therein). For a given 2-graph $G$, the nullity of $G$ is defined as the nullity of its adjacency matrix. Here we extend  the notion of nullity from graphs to hypergraphs. The \textit{nullity} of an $r$-graph $\mathcal{H}$, denoted $\eta(\mathcal{H})$, is defined to be the multiplicity of the eigenvalue 0 as a root of the characteristic polynomial of $\mathcal{H}$.  A well-known result on the matching number of 2-trees and its corresponding nullity is as follows.

\begin{theorem}[\cite{theorygraphspectra}]
If $G$ is a 2-tree on $n$ vertices,  then $\eta(G)=n-2\nu(G)$, where $\nu(G)$ is the size of a maximum matching in $G$ and is called the matching number of $G$.
\end{theorem}

The above result for 2-trees can be generalized to $r$-trees according to the following result, which is an immediate consequence of  Theorem~\ref{thm_main}.

\begin{coro}
For any $r$-tree $\mathcal{T}$ with $r\geq 2$, we have
\[
\eta(\mathcal{T})=\sum_{H\sqsubseteq \mathcal{T}}a_H(|H|-r\nu(H)),
\]
where the sum is over all connected subgraphs $H$ of $\mathcal{T}$.
\end{coro}

\subsection*{Acknowledgement.}  H. Li is supported by National Natural Science Foundation of China (No  12161047) and CSC. L. Su is supported by National Natural Science Foundation of China (No  12061038) and CSC.  Dr. Fallat's research was supported in part by an NSERC Discovery Research Grant, Application No.: RGPIN-2019-03934.
The work was done while H. Li  and L. Su  visited the   University of Regina to which thanks are given by Li and Su for their hospitality and support.


\begin{thebibliography}{99}



 \bibitem{BaoFanWang}
  Y. Bao, Y. Fan, Y. Wang, and M. Zhu, A combinatorial method for computing characteristic polynomials
of starlike hypergraphs.  {\em J. Algebraic Comb.} 51 (2018) 589-616.

\bibitem{Berge-hypergraph}
C. Berge, Graphs and Hypergraphs, North-Holland Mathematical Library, 2nd Ed., Vol. 6, North-
Holland, 1976.


\bibitem{Biggs-criticalgroup}
N. L. Biggs,
Chip-firing and the critical group of a graph.
{\em J. Algebraic Combin.} 9 (1999), no. 1, 25-45.


\bibitem{BjoLovaShor}
A.  Bj\"orner,  L. Lov\'asz, and P.W. Shor,  Chip-firing games on graphs. {\em Eur. J. Comb.} 12(4) (1991), 283-291.

 \bibitem{Bretto}
A. Bretto, Hypergraph Theory: An Introduction. Springer, 2013.

\bibitem{ChenBu}
L. Chen and C.  Bu,
A reduction formula for the characteristic polynomial of hypergraph with pendant edges.
{\em Linear Algebra Appl.} 611 (2021), 171-186.

%\bibitem{chung_dimacs_93}
%F. R. K. Chung,
%The Laplacian of a hypergraph,
%in: Expanding Graphs (Princeton, NJ, 1992),
%DIMACS Series in Discrete Mathematics and Theoretical Computer Science 10 (American Mathematical Society, Providence, RI, 1993), pp.~21-36.


\bibitem{ClarkCooper18}
G. J. Clark and J. Cooper,
On the adjacency spectra of hypertrees.
{\em Electron. J. Combin.} 25 (2018), Paper No. 2.48, 8 pp.



\bibitem{ClarkCoop21}
 G. J. Clark and J. Cooper, A Harary-Sachs theorem for hypergraphs. {\em J. Combin. Theory Ser. B} 149 (2021), 1-15.

\bibitem{CoopDut12}
J. Cooper and A. Dutle, Spectra of uniform hypergraphs.  {\em Linear Algebra Appl.}  436 (2012), 3268-3292.


\bibitem{sandpilegroupdualgraph}
R. Cori and D. Rossin,
On the sandpile group of dual graphs.
{\em European J. Combin.} 21 (2000), no. 4, 447-459.



\bibitem{CoxLittleShea}
  D. A. Cox, J. Little,  and D. O'Shea, Using Algebraic Geometry, Graduate Texts in Mathematics, Vol. 185,
second edition, Springer, New York, 2005.

\bibitem{CoxLittleShea-ideals}
D. A. Cox, J. Little, and D. O'Shea,   Ideals, varieties, and algorithms.   Undergraduate Texts in Mathematics. Springer, Cham, 2015.

\bibitem{theorygraphspectra}
D. Cvetkovi\'c, M. Doob,  and H. Sachs. Spectra of graphs: theory and application. Academic Press,
1980.



\bibitem{AlgComb}
C. D. Godsil, Algebraic combinatorics. Chapman and Hall Mathematics Series. Chapman  \& Hall, New York, 1993.


\bibitem{Lim05} L.  Lim,
Singular values and eigenvalues of tensors: a variational approach. In: Proceedings of the IEEE International Workshop on Computational Advances in Multi-Sensor Adaptive Processing (CAMSAP'05), vol.~1 (2005), 129-132.

\bibitem{qi05}  L. Qi,  Eigenvalues of a real supersymmetric tensor. {\em J. Symb. Comput.} 40 (2005), 1302-1324.

\bibitem{QiLuo-2017}
L. Qi and Z. Luo,  Tensor Analysis: Spectral Theory and Special Tensors. SIAM, 2017.

\bibitem{ShaoQiHu}
 J.  Shao, L. Qi, and S. Hu, Some new trace formulas of tensors with applications in spectral hypergraph
theory.  {\em Linear Multilinear Algebra} 63 (2015), 971-992.



\bibitem{SuKLS}
L. Su, L. Kang and H. Li, E. Shan,
The matching polynomials and spectral radii  of uniform supertrees.  {\em Electron. J. Combin.} 25(4) (2018), \#P4.13.


\bibitem{chip-firing}
C. J. Klivans,
The   Mathematics of Chip-Firing.  CRC Press, 2018.




\bibitem{WangGuo}
Z. Wang and J. Guo,  A sharp upper bound of the nullity of a connected graph in terms of order and maximum degree. {\em Linear Algebra Appl.} 584 (2020), 287-293.


\bibitem{Zhang_17}
W. Zhang, L. Kang, E. Shan, and Y. Bai,  The spectra of uniform hypertrees.  {\em Linear Algebra Appl.}  533 (2017), 84-94.

\bibitem{complete3-graph}
Y. Zheng, The characteristic polynomial of the complete 3-uniform hypergraph. {\em Linear Algebra Appl.} 627 (2021), 275-286.


\bibitem{zykov}
A. A. Zykov,  Hypergraphs. Uspehi Mat. Nauk 29 (1974), no. 6 (180), 89-154.

\end{thebibliography}
\end{document}